\documentclass[11pt]{article}

\usepackage[utf8]{inputenc}

\usepackage{csquotes}
\usepackage{cite}

\usepackage[english]{babel}
\usepackage{amsmath}
\usepackage{amsthm}
\usepackage{amssymb}

\usepackage[titles]{tocloft}
\usepackage{geometry}
\usepackage{hyperref}

\usepackage{dsfont}
\usepackage{wasysym}

\usepackage{tikz}
\usepackage{tikz-cd}

\usepackage[titles]{tocloft}
\usepackage{hyperref}

\newtheorem{thm}{Theorem}[section]
\newtheorem*{thm*}{Theorem}
\newtheorem{theorem}{Theorem}

\newtheorem{corA}[theorem]{Corollary}

\newtheorem{conj}[thm]{Conjecture}

\newtheorem{warning}[thm]{Warning}
\newtheorem{proposition}[thm]{Proposition}

\newtheorem{lem}[thm]{Lemma}
\newtheorem{cor}[thm]{Corollary}

\theoremstyle{definition}

\newtheorem{defn}[thm]{Definition}

\theoremstyle{remark}
\newtheorem{remark}[thm]{\textbf{Remark}}
\newtheorem{construction}[thm]{\textsc{Construction}}
\newtheorem{example}[thm]{\textsc{Example}}

\newtheoremstyle{note}
{0pt}
{0pt}
{}
{}
{\bfseries}
{.}
{.5em}
{}

\theoremstyle{note}
\newtheorem{nrem}[thm]{}

\newcommand{\IN}{\mathbb{N}}
\newcommand{\IZ}{\mathbb{Z}}

\newcommand{\IR}{\mathbb{R}}
\newcommand{\IC}{\mathbb{C}}

\newcommand{\IS}{\mathbb{S}}
\newcommand{\IT}{\mathbb{T}}
\newcommand{\II}{\mathbb{I}}

\newcommand{\Cat}{\mit{Cat}}

\newcommand{\Cati}{\Cat_\infty}
\newcommand{\Catit}{\Cati^\ot}

\newcommand{\Top}{{\mit{T\!op}}}

\newcommand{\Diff}{{\op{Diff}}}
\newcommand{\Emb}{\mi{Emb}}

\newcommand{\Sets}{{\mathit{Set}}}
\newcommand{\Spc}{{\mathit{Spc}}}

\newcommand{\op}[1]{\operatorname{#1}}
\newcommand{\id}{\op{id}}

\newcommand{\pr}{\op{pr}}

\newcommand{\Tr}{\op{Tr}}
\newcommand{\Map}{\op{Map}}
\newcommand{\Hom}{\op{Hom}}
\newcommand{\Aut}{\op{Aut}}
\newcommand{\End}{\op{End}}

\newcommand{\colim}{\operatornamewithlimits{colim}}

\newcommand{\obj}{\op{obj}}

\newcommand{\Sc}{\mi{Sc}}

\newcommand{\eps}{\varepsilon}
\newcommand{\ga}{\alpha}
\newcommand{\gb}{\beta}
\newcommand{\gd}{\delta}
\newcommand{\gD}{\Delta}
\newcommand{\gc}{\gamma}
\newcommand{\gC}{\Gamma}
\newcommand{\gi}{\iota}

\newcommand{\gl}{\lambda}
\newcommand{\gL}{\Lambda}

\newcommand{\gS}{\Sigma}
\newcommand{\gt}{\theta}

\newcommand{\gO}{\Omega}
\newcommand{\gp}{\varphi}

\newcommand{\mi}[1]{\mathit{#1}}

\newcommand{\mc}[1]{\mathcal{#1}}
\newcommand{\mcA}{\mc{A}}

\newcommand{\mcC}{\mc{C}}
\newcommand{\mcD}{\mc{D}}
\newcommand{\mcE}{\mc{E}}

\newcommand{\mcL}{\mc{L}}

\newcommand{\mcP}{\mc{P}}

\newcommand{\mcZ}{\mc{Z}}
\newcommand{\sP}{\mc{P}}

\renewcommand{\d}[1]{\op{d}\!#1\;}

\newcommand{\inj}{\hookrightarrow}

\newcommand{\curveto}{\curvearrowright}

\newcommand{\ot}{\otimes}

\newcommand{\cd}{\bullet}
\newcommand{\1}{\mathds{1}}

\renewcommand{\c}[1]{\mathtt{#1}}

\newcommand{\cC}{\c{C}}
\newcommand{\cD}{\c{D}}
\newcommand{\cE}{\c{E}}

\newcommand{\CSS}{\c{CSS}}
\newcommand{\Fun}{\c{Fun}_\infty}

\newcommand{\Funo}[1]{\Fun^{\ot}\!\left(#1\right)}

\newcommand{\Vect}{\c{Vect}}

\newcommand{\Cob}{\mi{Cob}}
\newcommand{\Mfd}{\c{Mfd}}

\newcommand{\Bordn}[1]{\c{Bord}_{ #1 }}
\newcommand{\Bone}{\Bordn{1}}

\newcommand{\Bor}{\Bordn{1}^{\op{or}}}
\newcommand{\Bfrn}[1]{\Bordn{#1}^{\op{fr}}}
\newcommand{\Born}[1]{\Bordn{#1}^{\op{or}}}
\newcommand{\Bormn}[1]{\Bordn{#1}^{\op{or},\op{m}}}

\newcommand{\Bfr}{\Bone^{\op{fr}}}

\newcommand{\Borm}{\Bone^{\op{or}, \op{m}}}

\newcommand{\ev}{\c{ev}}
\newcommand{\co}{\c{co}}

\newcommand{\fd}{{\op{fd}}}
\newcommand{\EPi}{\mi{EP}^\infty\!}

\newcommand{\icat}{$(\infty,1)$-category{}}
\newcommand{\icats}{{$(\infty,1)$-categories}}

\newcommand{\qand}{\quad \text{and} \quad}

\newcommand{\blank}{\underline{\ \ }}
\newcommand{\gle}[1]{\langle #1 \rangle}

\newcommand{\cube}{(-1,1)^\infty}

\usepackage{mathrsfs}

\newcommand{\cc}[1]{\mathtt{#1}}

\title{The space of traces in symmetric monoidal infinity categories}
\author{Jan Steinebrunner}

\renewcommand{\EPi}{\op{Free}_{E_\infty}\!}
\renewcommand{\Spc}{\c{Spc}}
\renewcommand{\Cati}{\c{Cat}_\infty}
\renewcommand{\sP}{\c{P}}
\newcommand{\Sp}{\c{Sp}}

\renewcommand{\icat}{$\infty$-category}
\renewcommand{\icats}{{$\infty$-categories}}

\newcommand{\CoZ}{\mi{Cob}_1^\IZ}
\newcommand{\CoN}{\mi{Cob}_1^\IN}

\newcommand{\Conf}{\mi{Conf}}

\newcommand{\TTo}{\mathscr{T}^r}
\newcommand{\eTTo}{\mathscr{T}}
\newcommand{\pTTo}{\mathscr{T}^{(r)}}
\newcommand{\TTi}{\mathscr{T}_\infty^r}
\newcommand{\eTTi}{\mathscr{T}_\infty}
\newcommand{\pTTi}{\mathscr{T}_\infty^{(r)}}
\newcommand{\gT}{\Theta}

\begin{document}


\maketitle

\begin{abstract}
    We define a \emph{tracelike transformation} to be a natural family of conjugation
    invariant maps $T_{x,\cC}:\hom_\cC(x, x) \to \hom_\cC(\1,\1)$ 
    for all dualisable objects $x$ in any symmetric monoidal \icat\ $\cC$.
    This generalises the trace from linear algebra that assigns a scalar $\Tr(f) \in k$
    to any endomorphism $f:V \to V$ of a finite-dimensional $k$-vector space.
    Our main theorem computes the moduli space of tracelike transformations
    using the one-dimensional cobordism hypothesis
    with singularities.
    
    As a consequence we show that the trace $\Tr$ can be uniquely extended 
    to a tracelike transformation up to a contractible space of choices.
    This allows us to give several model-independent characterisations of 
    the $\infty$-categorical trace.
    Restricting our notion of tracelike transformations 
    from endomorphisms to automorphisms
    we in particular recover a theorem of To\"en and Vezzosi.
    
    Other examples of tracelike transformations are for instance given by 
    $f \mapsto \Tr(f^n)$. Unlike for $\Tr$ the relevant connected
    component of the moduli space is not contractible,
    but rather equivalent to $B\IZ/n\IZ$ or $BS^1$ for $n=0$.
    As a result we obtain a $\IZ/n\IZ$-action on $\Tr(f^n)$
    as well as a circle action on $\Tr(\id_x)$.
\end{abstract}


\section{Introduction}\label{sec:Introduction}

The trace, originally defined in the context of linear algebra, 
admits a well-known generalisation to arbitrary symmetric monoidal categories. 
Often the trace can be thought of as measuring `fixed-points' of an 
endomorphism, most prominently in the case of the stable homotopy category,
where it computes the Lefschetz fixed-point number, as we explain below.
Recently this has been vastly generalised to indexed, relative, and 
equivariant settings using bicategories.
We refer the reader to Ponto and Shulman's paper \cite{PS14}
for an introduction to traces from this perspective and further references.

For bordism categories the trace behaves like the `closing up' 
operation that sends a braid to its associated link. 
In the context of topological quantum field theories
this corresponds to taking the so-called `state-sum',
see for example Stolz and Teichner's work \cite{ST12}.
The example of bordism categories plays an important role later since,
by the cobordism hypothesis, bordism categories are the universal 
setting for taking traces.
This generalises the strategy of To\"en and Vezzosi \cite{TV15}.

Another example, which highlights the importance of treating the trace 
$\infty$-categorically, is the derived Morita category 
of ring spectra and bimodules. Here the trace is given by topological
Hochschild homology (THH). We will return to this example in the final section
where we show that in this case Theorem \ref{thm:As} induces 
a circle action on $\op{THH}(R)$ and a $\IZ/n\IZ$-action on
$\op{THH}(R; M \ot_R \dots \ot_R M)$.

\subsection*{Traces in symmetric monoidal $1$-categories}
Before we consider \icats, let us recall the trace in $1$-categories. 
The trace of an endomorphism $f:V\to V$ on a finite-dimensional 
$k$-vector space $V$ may be defined as 
\[
    \Tr(f):=\sum_{i\in I} \gb^i(f(b_i)) \in k,
\]
where $(b_i)_{i\in I}$ is a basis of $V$ and $(\gb^i)_{i\in I}$ 
is the dual basis of $V^*:=\hom_{\Vect_k}(V,k)$ 
uniquely characterised by $\gb^i(b_j)=\gd_{ij}$. 
Let us define the evaluation and coevaluation of $V$ by
\[
    e:V^*\ot V \to k, (\ga,v) \mapsto \ga(v) 
    \qand c:k \to V \ot V^*, \gl \mapsto \gl \sum
                                              b_i \ot \gb^i
\]
and rewrite the trace as the composition
\[
    \Tr(f): k \xrightarrow{c} V \ot V^* 
    \xrightarrow{f \ot \id_{V^*}} V \ot V^* 
    \xrightarrow{\cc{swap}_{V,V^*}} V^* \ot V 
    \xrightarrow{e} k.
\]
This definition generalises to all symmetric monoidal 
categories $(\mcC,\ot,\1)$:
an object $x\in \mcC$ is called \emph{dualisable}
if there is a dual $y\in \mcC$ with an evaluation $e:y\ot x\to\1$
and a coevaluation $c:\1 \to x\ot y$ satisfying:
\begin{equation}\label{eqn:snake}
    \id_x = (\id_x \ot\, e) \circ (c \ot \id_x) \qand \id_y 
    = (e \ot \id_y) \circ (\id_y \ot\, c).
\end{equation}
The data $(y,e,c)$ is essentially unique, if it exists.
For such $x$ the trace of $f:x \to x$ is defined as:
\begin{align*}
    \Tr_x(f) := e \circ \cc{swap}_{x,y} \circ (f \ot \id_y) \circ c 
        \in \hom_\mcC(\1,\1).
\end{align*}
\begin{defn}
    In analogy to the category of $k$-vector spaces we will refer 
    to endomorphisms of the unit object as \emph{scalars}.
    They form a commutative monoid and taking scalars defines a functor:
    \[
        \Sc: \Cat^\ot \to \mi{CMon}, \quad 
        \mcC \mapsto \Sc(\mcC) :=\hom_\mcC(\1,\1).
    \]
    To a symmetric monoidal functor $F:\mcC \to \mcD$ with 
    unitor $\gl_F:F(\1_\mcC) \cong \1_\mcD$ this assigns the map
    $\Sc(\mcC) \to \Sc(\mcD)$ sending $f$ to $\gl_F \circ F(f) \circ \gl_F^{-1}$.
\end{defn}

This generalisation of the trace retains the cyclicity of the trace
that is well-known in the context of linear algebra \cite[Proposition 2.4]{PS14}:
when $f:x \to z$ and $g:z \to x$ are morphisms 
between dualisable objects, then
\[
    \Tr_x(g \circ f) = \Tr_y(f \circ g).
\]
The main advantage of having such a general definition is that 
we are now able to compare traces in different symmetric monoidal 
categories \cite[Proposition 6.2]{PS14}. Let $F:\mcC \to \mcD$ be a symmetric monoidal functor,
$x \in \mcC$ dualisable, and $f:x \to x$ an endomorphism.
Then $F(x)$ is dualisable and the trace of $F(f)$ is
\begin{align}\label{eq:natsymmon}
    \Tr_{\mcD,F(x)}(F(f)) = F(\Tr_{\mcC,x}(f)).
\end{align}
An interesting example is the stable homotopy category $\mcC=h_1\Sp$
with the smash-product as symmetric monoidal structure.
Its scalars are homotopy classes of maps $\IS \to \IS$ and
these are classified by integers.
For a finite CW-complex $X$ the suspension spectrum $\gS^\infty_+ X$
is a dualisable object of $h_1 \Sp$
with dual the \emph{Spanier-Whitehead dual} of $X$.
A continuous map $f:X\to X$ induces an endomorphism $\gS^\infty_+ f$; 
its trace $\Tr(\gS^\infty_+ f) \in \Hom_{h_1\Sp}(\IS,\IS) \cong \IZ$ 
is the so-called Lefschetz-number $\gL(f)$ of $f$,
which can be understood in terms of fixed-points of $f$.
The functor that takes a spectrum to its cohomology with coefficients
in a field is symmetric monoidal by the K\"unneth-theorem. 
The naturality (\ref{eq:natsymmon}) implies that we can understand 
$\gL(f)$ in terms of the action $f$ has on the homology of $X$.
This is the Lefschetz fixed-point theorem. 

\subsection*{Axiomatic description: tracelike transformations}
Our axiomatisation of the properties of the trace is loosely based 
on Kelly and Laplaza's notion of a `trace function' from \cite{KL80}.
We discuss how they are related below.
The main difference is that we will study compatible choices of traces for 
\emph{all} symmetric monoidal categories at the same time:

\begin{defn}\label{defn:tt}
    A \emph{tracelike transformation} $T$ is a family of maps
    \[
        T_{\mcC,x}: \End_\mcC(x) \to \Sc(\mcC) = \End_\mcC(\1) 
    \]
    for all symmetric monoidal categories $\mcC$ and dualisable objects $x \in \mcC$
    satisfying the following axioms:
    \begin{itemize}
        \item \emph{Conjugation invariance:} 
            $T_y(\gp \circ e \circ \gp^{-1}) = T_x(e)$ 
            for all endomorphisms $e:x \to x$ and isomorphisms $\gp:x \cong y$.
        \item \emph{Naturality:}
            $ T_{\mcD,F(x)}(F(e)) = \gl_F \circ F(T_{\mcC,x}(e)) \circ \gl_F^{-1}$ 
            for every symmetric monoidal functor $F:\mcC \to \mcD$ 
            with unitor $\gl_F:F(\1_\mcC) \cong \1_\mcD$, 
            dualisable object $x \in \mcC$, and endomorphism $e:x \to x$.
    \end{itemize}
    A \emph{restricted tracelike transformation} is a family of maps
    $ T_{\mcC,x}: \Aut_\mcC(x) \to \Sc(\mcC) $ 
    satisfying the same axioms.
    We denote the set of tracelike transformations by $\eTTo$
    and the set of restricted tracelike transformations by $\TTo$.
\end{defn}

Since $\Sc(\mcC)$ is a commutative monoid for all symmetric monoidal categories $\mcC$ 
the sets $\eTTo$ and $\TTo$ inherit commutative monoid structures defined
by multiplying tracelike transformations pointwise.
Moreover, the multiplicative monoid $(\IN, \cdot)$ acts on these monoids
by maps $\mcP^n:\eTTo \to \eTTo$ defined as $\mcP^n(T)(e) := T(e^n)$.
Starting from the classical trace $\Tr \in \eTTo$ that we discussed above,
we can therefore construct, for every finite sequence of natural numbers
$k_1, \dots, k_n \in \IN$, a tracelike transformation:
\[
    \gT_x^{k_1,\dots,k_n}(e) 
    = \Tr_x(e^{k_1}) \circ \dots \circ \Tr_x(e^{k_n}).
\]
In fact, if we only want $\gT$ to be a restricted tracelike transformation,
we can allow the $k_i$ to be integers.

We now describe an equivalent definition of tracelike transformations:
Let $\Cat_1^\ot$ denote the category with objects symmetric monoidal categories
and morphisms natural isomorphism classes of symmetric monoidal functors.
We define a functor $E^\fd:\Cat_1^\ot \to \Sets$ by  
\[
    E^\fd(\mcC) = \coprod_{x \in \mcC \text{ dualisable}} \End_\mcC(x) \big/\sim
\]
where $(a:x \to x) \sim (b:y \to y)$ whenever there is an isomorphism 
$\gp:x \to y$ with $\gp \circ a \circ \gp^{-1} = b$.
Similarly, define $A^{\fd}(\mcC) \subset E^\fd(\mcC)$ 
as the subset of conjugacy classes represented by automorphisms.
By construction, a tracelike transformation is simply a natural transformation 
$E^\fd \Rightarrow \Sc$ of functors $\Cat_1^\ot \to \Sets$.
Similarly, a restricted tracelike transformation 
is a natural transformation $A^\fd \Rightarrow \Sc$.

Note that our $E^\fd(\mcC)$ is similar to Kelly and Lapaza's 
\emph{set of cycles} $[\mcD]$, which they define in \cite{KL80} for any category $\mcD$. 
If we let $\mcC^\fd \subset \mcC$ denote the full subcategory 
on the dualisable objects and $(\mcC^\fd)^\simeq \subset \mcC^\fd$ the maximal subgroupoid,
then there is a canonical surjection $E^\fd(\mcC) \to [\mcC^\fd]$,
and a canonical bijection $A^\fd(\mcC) \to [(\mcC^\fd)^\simeq]$.
The crucial difference is that they require a cyclicity condition
$T(f \circ g) = T(g \circ f)$ for any two morphisms $f:x \to y$ and $g:y \to x$,
whereas this only follows from our conjugation invariance if one of $f$ and $g$
is an isomorphism.

\subsection*{$1$-categorical classification of tracelike transformations}
We will show that the $\gT^{k_1,\dots,k_n}$ are indeed the only examples 
of tracelike transformations: 
a variant of the cobordism hypothesis in dimension $1$ implies
that certain cobordism categories $\CoN \subset \CoZ$ are freely generated 
by a dualisable object $*_+$ and $\ga: *_+ \to *_+$ 
an endomorphism or automorphism, respectively.
Hence there is a bijection between the set of tracelike transformations
and the scalars of $\mi{Cob}_1^\IN$.
These can be described as $\Mfd_1^\IN/\Diff^+$:
diffeomorphism classes of closed $1$-dimensional manifolds
labelled by natural numbers.
This in turn is in bijection with $\IN[x]$, 
the free commutative monoid on the set $\{1,x,x^2,\dots\}$.
In summary this leads to:
\begin{thm*}[see {\ref{prop:1class}}]
    There are compatible isomorphisms of commutative monoids
    \[
    \begin{tikzcd}
        \eTTo \ar[r, "\cong"] \ar[d]
        & \Sc(\mi{Cob}_1^\IN) \ar[r, "\cong"] \ar[d]
        & \Mfd_1^\IN/\Diff^+ \ar[r, "\cong"] \ar[d]
        & \IN[x] \ar[d] \\
        \TTo \ar[r, "\cong"] 
        & \Sc(\mi{Cob}_1^\IZ) \ar[r, "\cong"] 
        & \Mfd_1^\IZ/\Diff^+ \ar[r, "\cong"] 
        & \IN[x^{\pm 1}] 
    \end{tikzcd}
    \]
    and the tracelike transformation $\gT^{k_1,\dots,k_n}$ is sent to 
    $\sum_{i=1}^n x^{k_i}$.
\end{thm*}

\subsection*{Traces in \icats}
If $\cC$ is a symmetric monoidal \icat\, then its scalars $\Sc(\cC)=\hom_\cC(\1,\1)$ 
are no longer just a set, but a space.
In fact, they naturally carry the structure of an $E_\infty$-space.
Consider for instance the \icat\ of spectra $\cC =\Sp$:
the endomorphisms of the unit object are:
$\Sc(\Sp) = \hom_{\Sp}(\IS,\IS) = \gO^\infty \IS$.
For a finite CW-complex $X$ the $1$-categorical trace on the homotopy category
$h_1\Sp$ can be promoted to a map of spaces:
\[
    \Tr: \hom_{\Sp}(\gS^\infty_+ X, \gS^\infty_+ X) \to \gO^\infty \IS.
\]
In definition \ref{defn:functors} and \ref{defn:functors-symmon}
we will construct $\infty$-functors $\Sc$, $\mcE^\fd$, and $\mcA^\fd$ 
from the $\infty$-category of symmetric monoidal $\infty$-categories $\Catit$
to the \icat\ of spaces $\Spc$.
For every symmetric monoidal \icat\ $\cC$ and dualisable object $x \in \cC$
there are compatible maps $\hom_\cC(x, x) \to \mcE^\fd(\cC)$
and $\mi{hAut}_\cC(x) \to \mcA^\fd(\cC)$.
\begin{defn}
    An $\infty$-categorical \emph{tracelike transformation} 
    is a natural transformation $T:\mcE^{\fd} \Rightarrow \Sc$ of 
    $\infty$-functors $\Catit \to \Spc$.
    Similarly, a \emph{restricted tracelike transformation} 
    is a natural transformation $T:\mcA^{\fd} \Rightarrow \Sc$.
    We denote the space of (restricted) $\infty$-categorical 
    tracelike transformations by $\pTTi$.
\end{defn}
The $\infty$-functors $\Sc$, $\mcE^\fd$, and $\mcA^\fd$ recover their
$1$-categorical analogues on the homotopy category $h_1\cC$ of $\cC$
in the sense that $\pi_0 \Sc(\cC) \cong \Sc(h_1 \cC)$,
$\pi_0 \mcE^\fd(\cC) \cong E^\fd(h_1\cC)$, and
$\pi_0 \mcA^\fd(\cC) \cong A^\fd(h_1\cC)$, see \ref{lem:recov}.
For any tracelike transformation $T$, symmetric monoidal \icat\ $\cC$
and dualisable object $x \in \cC^\fd$ naturality with respect to the functor 
$\cC \to h_1\cC$ yields the following commutative diagram: 
\[
    \begin{tikzcd}
        \hom_\cC(x, x) \ar[r] \ar[d] 
        & \mcE^\fd(\cC) \ar[r, "T_\cC"] \ar[d] 
        & \Sc(\cC) \ar[d] \\
        \hom_{h_1\cC}(x, x) 
        \ar[r] 
        & E^\fd(h_1\cC) \ar[r, "T_{h_1\cC}"] 
        & \Sc(h_1\cC) 
    \end{tikzcd}
\]
Here we implicitly used the map $\eTTi \to \eTTo$ defined by sending 
an $\infty$-categorical tracelike transformation to its restriction 
to $1$-categories.
This map encodes how an $\infty$-categorical tracelike transformation
behaves on the level of connected components.

Our main theorem describes the homotopy-type of $\eTTi$ 
and the map to the discrete set $\eTTo$.
Generalising the commutative monoid structure on $\pTTo$
there is an $E_\infty$-algebra structure $\pTTi$.
For simplicity, we will only identify the underlying space.
To state the theorem, let $\EPi(X)$ denote underlying space 
of the free $E_\infty$-algebra on $X$.

\begin{theorem}[{\ref{thm:Ap}}]\label{thm:As}
    There is a commutative diagram of spaces:
    \[
        \begin{tikzcd}[row sep=tiny, column sep=small]
            \eTTi \ar[dd] \ar[rr, "\simeq"] \ar[rd] & &
            \EPi(BS^1 \amalg \coprod_{k\ge 1}B\IZ/k\IZ)\ar[dd] \ar[rd] & \\
            & \TTi \ar[dd] \ar[rr, "\simeq", near start] & &
            \EPi(S^1 \times BS^1 \amalg 
                    \coprod_{k\in \IZ\setminus\{0\}}B\IZ/k\IZ)\ar[dd] \\
            \eTTo \ar[rr, "\cong"] \ar[rd] & & \IN[x] \ar[rd] & \\
            & \TTo \ar[rr, "\cong"] & & \IN[x^{\pm 1}]
        \end{tikzcd}
    \]
    where the horizontal maps are equivalences.
    Moreover, the vertical maps identify the sets in the bottom layer
    as the set of connected components of the top layer: 
    $\eTTo \cong \pi_0 \eTTi$ and $\TTo \cong \pi_0 \TTi$.
\end{theorem}

\begin{warning}\label{warning}
    The reader should be warned that in identifying $\eTTi$ we use the
    \emph{cobordism hypothesis with singularities}, which was sketched in \cite{LurCH}, 
    but a full proof of which has not yet appeared in the literature.
    We will therefore be treating it as a conjecture and all our statements
    about $\eTTi$ are dependent on this conjecture.
    Note, however, that the analogous statements about the space of
    \emph{restricted} tracelike transformations $\TTi$ only use the 
    standard cobordism hypothesis in dimension $1$, which has been proved in detail. 
    (See \cite{LurCH} and \cite{Har12}.)
\end{warning}

Of particular interest is the homotopy-type of the moduli space 
of those $\infty$-categorical tracelike transformations 
that recover the classical trace on homotopy categories.
In their work on the derived Chern-character \cite{TV15} To\"en and Vezzosi 
show that for restricted tracelike transformations the space is contractible
and that therefore there is an essentially unique $\infty$-categorical
generalisation of the trace when applied to automorphisms.

In theorem \ref{thm:As} we use methods similar to theirs to compute 
the full homotopy-type of $\pTTi$. 
Their result can now be read off by considering
the fibre of $\TTi \to \TTo$ over the classical trace.
Since we study $\eTTi$ as well we can now generalise their result,
removing the artificial restriction to automorphisms.
\begin{corA}[{Generalising \cite[Théorème 3.18]{TV15}, \ref{cor:Bp}}]\label{cor:Bs}
    The space of (restricted) $\infty$-categorical tracelike transformations 
    that act as the $1$-categorical trace $\Tr$ on homotopy categories is contractible.
\end{corA}

Using our knowledge of the structure of $\pi_0\eTTi$ it is in fact not hard to see
that to uniquely characterise the $\infty$-categorical trace we only need to specify
its behaviour on the category of vector spaces.
\begin{corA}[\ref{cor:Cp}]\label{cor:Cs}
    Any $\infty$-categorical tracelike transformation
    whose value on the category of complex vector spaces agrees 
    with the trace from linear algebra is canonically equivalent
    to the $\infty$-categorical trace.
\end{corA}

More informally, there is a unique extension of the trace 
from linear algebra to a family of maps
\[
    \Tr_{(x,\cC)}: \End_\cC(x) \to \hom_\cC(\1,\1)
\]
for any symmetric monoidal \icat\ $\cC$ and any dualisable object 
$x\in\cC$ while preserving the conjugation invariance of the trace 
and its naturality with respect to symmetric monoidal functors.
    
We can also give an alternative characterisation of $\Tr$ 
as the unique \emph{generating tracelike transformation}.
This characterisation is purely categorical and does not
require one to first define a trace for vector spaces.

\begin{defn}
    The monoid $(\IN,\cdot)$ acts on $\eTTo$ and $\eTTi$ by
    taking powers of any morphism before applying the tracelike transformation 
    $\mcP^n(T)(a) := T(a^n)$.
    We call a tracelike transformation $T \in \eTTi$ \emph{generating}
    if the monoid $\pi_0 \eTTi$ is generated by the orbit of $T$ under the $\IN$-action.
\end{defn}

In other words, $T$ is generating if every other tracelike transformation 
$S \in \eTTi$ is equivalent to one of the form 
$\mcP^{k_1}(T) \circ \dots \circ \mcP^{k_n}(T)$.

\begin{corA}[{\ref{cor:Dp}}]\label{cor:Ds}
    The space of generating tracelike transformations in $\eTTi$ 
    is contractible and its image in $\eTTo$ is the usual trace.
\end{corA}

\subsection*{Notation}
We will assume that the reader has a convenient model of
$(\infty,1)$-categories at hand.
In this paper we will be working in the context of Joyal's 
quasicategories, but really any equivalent Cartesian closed
$\infty$-cosmos in the sense of Riehl and Verity will do.
We will refer to these $(\infty,1)$-categories as \icats\ 
and to morphisms between them as $\infty$-functors, 
or sometimes just as functors.

    We write $\Cati$ for the \icat\ of \icats\ and
    $\Spc \subset \Cati$ for the full subcategory of $\infty$-groupoids,
    which we will refer to synonymously as `spaces'.
    The $1$-category of topological spaces will be denoted by $\Top$.
    There is a functor $\Top \to \Spc$ that 
    `forgets the point-set information', 
    it sends a topological space to its $\infty$-groupoid of paths.

    For \icats\ $\cC$, $\cD$, $\cE$ we denote the \icat\ 
    of functors from $\cC$ to $\cD$ by $\Fun(\cC,\cD)$ 
    and the maximal subgroupoid of $\cE$ by $\cE^\sim \in \Spc$.
    For objects $a,b\in \cE$ the space of morphisms 
    from $a$ to $b$ is $\hom_\cE(a,b)$.
    In the case $\cE=\Cati$ the space of functors from $\cC$ to $\cD$ 
    is $\hom_{\Cati}(\cC,\cD)\simeq (\Fun(\cC,\cD))^\sim$.

    The \icat\ of (simplicial) presheaves on $\cC$ is 
    $\sP(\cC) := \Fun(\cC^{op}, \Spc)$.
    The Yoneda embedding will be denoted by $Y: \cC \to \sP(\cC)$.

\subsection*{Structure of the paper}
We begin by recalling complete Segal spaces 
and other $\infty$-categorical tools in section $2$,
where we also define the functors $\mcE^\fd$ and $\mcA^\fd$.
Then, in section $3$, we define concrete models of the one-dimensional 
bordism category with and without marked points and 
show that they define symmetric monoidal \icats.
Using these we formulate the cobordism hypothesis,
as well as a variant with singularities.
In section $4$ we complete the proof of the classification result 
for $1$-categorical tracelike transformations.
Section $5$ contains the homotopy-theoretic computations 
and the proofs of the main theorems.
In the final section we discuss how non-contractible connected components 
of the moduli space $\eTTi$ induce group actions on the trace.

\subsection*{Acknowledgements}
I would like to thank my advisor Ulrike Tillmann for introducing me 
to the world of bordism categories 
and for her support throughout the various stages of this paper.
My thanks also goes to Luciana Basualdo Bonatto for her comments.
I would also like to thank the referee for encouraging me to include the case of
un-restricted tracelike transformations when I first submitted the paper
and for their patience during the reviewing process.

I am very grateful for the support by St. John’s College, Oxford through the 
\emph{Ioan and Rosemary James Scholarship}, and the EPSRC grant no.\ 1941474.


\section[Modelling symmetric monoidal infinity categories]
{Modelling symmetric monoidal \icats}\label{sec:Modelling}
Much of the literature on the $\infty$-categorical structure 
of the cobordism category, in particular Lurie's work on the cobordism 
hypothesis \cite{LurCH}, is formulated in terms of $\gC$-objects in
complete Segal spaces.
In this section we recall how to relate this approach to the more 
standard theory of quasicategories, by giving a model independent
description of complete Segal spaces in the language of \icats.
Once this is established we make precise the functors $\mcL^\fd$,
$\mcA^\fd$, and $\Sc$ from the introduction.

\subsection{Complete Segal spaces}

\begin{nrem}
    We let $\Cat_1$ denote the $1$-category of $1$-categories.
    It admits a functor to the \icat\ of \icats\ $\Cat_1 \to \Cati$;
    in the quasicategory model this sends a category to its nerve.
    Write $\cc{Cat}_1 \subset \Cati$ for the essential image of the inclusion.
    We will think of this as the \icat\ of $1$-categories.
    This in fact is also a $(2,1)$-category:
    the hom-spaces of $\cc{Cat}_1$ are the $1$-groupoids 
    of $1$-functors and natural isomorphisms:
    \[
        \hom_{\cc{Cat}_1}(\mcC, \mcD) 
        \cong \left(\mi{Fun}_1(\mcC,\mcD)\right)^\sim.
    \]
    The inclusion $\cc{Cat}_1 \to \Cati$ has a left-adjoint,
    the homotopy-category functor 
    \[
        h_1:\Cati \to \cc{Cat}_1
    \]
    and in this sense $\cc{Cat}_1$ is a localisation of $\Cati$.
\end{nrem}

\begin{nrem}
    We define the simplex category as the full subcategory 
    $\gD \subset \cc{Cat}_1$ generated by the partially ordered sets 
    $[n]= \{0 \le \dots \le n\}$ thought of as categories, for $n\ge 0$.
	Using the embedding $I:\gD \to \Cati$ we obtain 
	\[
		N: \Cati \xrightarrow{Y} \sP(\Cati) 
            \xrightarrow{I^*} \sP(\gD).
	\]
    Theorem \ref{thm:CSS} states that this functor is fully faithful. 
    This is a way of saying that the objects $[n]$ generate 
    $\Cati$ \emph{strongly} under colimits.
    The essential image of $N$ are the complete Segal spaces:
\end{nrem}

\begin{defn}
    A simplicial space $X\in\sP(\gD)$ satisfies the \emph{Segal condition}
    if for all $n \ge 2$ the map 
    \[
        (\gl_0^*, \dots, \gl_{n-1}^*): X_n \longrightarrow X_1 \times_{X_0} \dots \times_{X_0} X_1
    \]
	induced by $\gl_i:[1] \to [n]$ with $\gl_i(k) = i+k$
	is an equivalence.
\end{defn}

\begin{defn}\label{defn:complete}
	Let $\II\in\Cat_1$ be the contractible groupoid with two objects 
    and $*\in \Cat_1$ the discrete category with one object.
    Write $N(\II)$ and $N(*)$ for the simplicial sets 
    that are the nerves of these categories and interpret them as 
    simplicial spaces that are discrete in every layer.%
    \footnote{This construction defines a functor
        $N:\Cat_1 \to \sP(\gD)$,
        but this functor does not preserve categorical equivalences:
        $\II$ and $*$ are equivalent, but $N(\II)$ and $N(*)$ are not.
        The functor $N:\Cati \to \sP(\gD)$ preserves equivalences 
        by construction, and simplicial spaces in its image will be 
        complete as they 
        `cannot see the difference between $\II$ and $*$'.
    }
    A simplicial space $X \in \sP(\gD)$ is called \emph{complete} 
    if the natural map
	\[
		X_0 \cong \hom_{\sP(\gD)}(N(*), X) 
            \to \hom_{\sP(\gD)}(N(\II), X) 
	\]
	coming from the $1$-functor $\II \to *$ is an equivalence. 
\end{defn}

\begin{defn}
	The \icat\ of \emph{complete Segal spaces} is defined as the 
    full subcategory $\CSS \subset \sP(\gD)$ spanned by the objects 
    that are complete and satisfy the Segal condition.
\end{defn}

\begin{remark}\label{rem:Rezk-complete}
    Given a Segal space $X$ there also is a simpler characterisation 
    of the completeness condition due to Rezk. By \cite[Theorem 6.2]{Rez01}
    the space $\hom_{\sP(\gD)}(N(\II), X)$ is equivalent to the subspace
    $X_1^{eq} \subset X_1$ on those $1$-simplices that represent 
    an isomorphism in the homotopy category $h_1 X$,
    which we describe in \ref{nrem:h-cat}.
    Moreover, $X_1^{eq} \subset X_1$ is always a union of connected components.
    Therefore a Segal space is complete if and only if $s_0:X_0 \to X_1^{eq}$
    is an equivalence.
\end{remark}

\begin{thm}[\!{\cite{Ber07}}]\label{thm:CSS}
	The functor $N$ takes values in complete Segal spaces 
    and induces an equivalence
	\[
		N: \Cati \xrightarrow{ \simeq } \CSS.
	\]
\end{thm}
\begin{proof}[Proof by citation]
    It is not difficult to see that for an \icat\ $\cC$ the nerve 
    $N\cC$ indeed satisfies the Segal and the completeness condition.
    To prove that $N$ induces an equivalence is more difficult.

    The first result of this type was \cite{Ber07},
    but there the model used for $\Cati$ was simplicial categories.
    Joyal and Tierney show in \cite{JT07} that the model categories 
    of quasicategories and complete Segal spaces are Quillen-equivalent.
    The $\infty$-categorical statement is an immediate consequence 
    of their result,
    see \cite[Corollary 4.3.16]{LurGW}.
\end{proof}

\begin{nrem}\label{nrem:recover}
	It is important to understand how standard constructions
    in $\Cati$ change under the equivalence to $\CSS$.
    Let $X = N(\cC)$ be the complete Segal space of some 
    \icat\ $\cC$.
	By definition, the levels of $X$ are of the form
	\[
		X_n = \hom_{\Cati}([n], \cC) = (\Fun([n],\cC))^\sim.
	\]
	In particular $X_0$ is the maximal subgroupoid $\cC^\sim$ of $\cC$ 
    and $X_1$ is the maximal subgroupoid of 
    the arrow category $\cC^{[1]}$.
	
	We will say that an object of $\cC$ is a functor $* \to \cC$,
    where $*$ is the terminal category.
	For two such objects $a,b:* \to \cC$ one can reconstruct 
    the hom-space $\hom_\cC(a,b)$ from the complete Segal space $X$ as:
	\[
		\hom_\cC(a,b) 
            \simeq \{a\} \times_{\cC^\sim} (\Fun([1], \cC))^\sim  
                \times_{\cC^\sim} \{b\} 
		= \{a\} \times_{X_0} X_1  \times_{X_0} \{b\} . 
	\]
	We can recover the composition, 
    up to inverting the weak equivalence $X_2 \to X_1 \times_{X_0} X_1$:
	\begin{align*}
		\hom_\cC(a,b)\times \hom_\cC(b,c) 
            &= (\{a\} \times_{X_0} X_1 \times_{X_0} \{b\}) 
                \times (\{b\} \times_{X_0} X_1 \times_{X_0} \{c\}) \\
		& \to \{a\} \times_{X_0} X_1 \times_{X_0} X_1 \times_{X_0} \{c\}
		\xleftarrow{\sim}  \{a\} \times_{X_0} X_2 \times_{X_0} \{c\}\\
		& \xrightarrow{d_1}  \{a\} \times_{X_0} X_1 \times_{X_0} \{c\}
            = \hom_\cC(a,c).
	\end{align*}
\end{nrem}
\begin{nrem}\label{nrem:h-cat}
	The above suffices to reconstruct the homotopy category
    $h_1 \cC$ of $\cC$ up to categorical equivalence:
    we define the set of objects as $O:= \pi_0 X_0$ and pick a section
    $o:\pi_0 X_0 \to X_0$ to interpret them in the above sense.
	For two $a,b \in O$ we define the morphism set as
	\[
        \Hom_\mcC(a,b) 
        := \pi_0 \left(\{o(a)\} \times_{X_0} X_1 \times_{X_0} 
             \{o(b)\}\right).
	\]
	The composition in $\mcC$ is constructed by taking $\pi_0$ 
    of what we did earlier. 
	Note that this erases the ambiguity coming from 
    inverting the equivalence.
\end{nrem}

\begin{remark}
    The construction given above has the disadvantage 
    that it is not functorial in $X$: 
    we need to make the unnatural choice of a section 
    $o:\pi_0 X_0 \to X_0$.
    In fact, we should not expect there to be a $1$-categorical 
    description since the $\infty$-functor 
    $h_1:\Cati \to \cc{Cat}_1$ does not factor through the $1$-category
    of $1$-categories $\Cat_1$.
\end{remark}

\subsection{Dualisability}
We recall the necessary definitions to talk about dualisable objects:
\begin{defn}[\!{\cite{Seg74}}]
	Segal's category $\gC$ is defined as a skeleton of the opposite category 
    of the category of finite pointed sets.
	The objects are $\gle{k} := \{*,1,\dots,k\}$ for $k \ge 0$
	and morphisms in $\gC^{op}$ are basepoint preserving maps.
	A functor $X:\gC^{op} \to \cC$ is called a 
    \emph{special $\gC$-object} in $\cC$ if
    $X(\gle{0})$ is a terminal object of $\cC$ and 
    for all finite pointed sets $A,B$ 
    the canonical maps $A \vee B \to A$ and $A \vee B \to B$ 
    induce an equivalence
	\[
		X(A\vee B) \xrightarrow{\simeq} X(A) \times X(B).
	\]
	Let $\CSS^\ot$ denote the \icat\ of special $\gC$-objects in $\CSS$.
	This can be thought of as a full subcategory of
    $\sP(\gD \times \gC)$.
	We write $X_n^{\gle{k}}$ for the value of 
    $X:\gD^{op} \times \gC^{op} \to \Spc$ on $([n],\gle{k})$.
\end{defn}

\begin{thm}\label{thm:smCSS}
	The \icat\ of special $\gC$-objects in $\cC$ is a model
    for the commutative monoid objects in $\cC$.
	In particular, the nerve functor $N: \Cati \simeq \CSS$ lifts to an 
    equivalence $\Catit \simeq \CSS^\ot$ between the \icat\ of 
    symmetric monoidal \icats\ and the \icat\ of special 
    $\gC$-objects in $\CSS$.
\end{thm}
\begin{proof}[Proof by citation]
    Commutative monoids in $\cC$ are by definition 
    $E_\infty$-algebras in $\cC$ with respect to the symmetric 
    monoidal structure coming from the Cartesian product.
    That these are the same as functors $\gC^{op} \to \cC$ 
    satisfying the `specialness condition' is for instance shown in 
    \cite[Proposition 2.4.2.5]{LurHA}.
\end{proof}

\begin{cor}\label{cor:h1N}
    The localisation-adjunction $h_1 \dashv I$ between $\Cati$
    and $\cc{Cat}_1$ lifts to a localisation-adjunction
    \[
        h_1^\ot:  \Catit\rightleftarrows \cc{Cat}_1^\ot : I^\ot.
    \]
\end{cor}
\begin{proof}
    There is an adjunction on the functor categories
    \[
        (h_1)_*: \Fun(\gC^{op}, \Cati) 
        \rightleftarrows \Fun(\gC^{op}, \cc{Cat}_1) :I_*
    \]
    and since both functors preserve the `specialness' of 
    $\gC$-objects this adjunction restricts to the full subcategories 
    $\Catit$ and $\cc{Cat}_1^\ot$.
    The functor $I_*$ is fully faithful because $I$ was and hence 
    $h_1^\ot$ is a localisation.
\end{proof}

\begin{nrem}
    Note that here $\cc{Cat}_1^\ot$ is by definition the \icat\ of 
    special $\gC$-objects in $\cc{Cat}_1$.
    It is a folklore theorem that this is equivalent to the 
    $(2,1)$-category of symmetric monoidal categories.
	For the readers convenience we will use this theorem and from now on 
    think of $h_1 X$ as a symmetric monoidal category.
    However, it is worth remarking that we could equally well work 
    with (special) $\gC$-categories.
\end{nrem}

\begin{defn}\label{defn:dual}
    An object $x$ in a symmetric monoidal \icat\ $\cC$ is called 
    \emph{dualisable} if $x$ is dualisable as an object of the symmetric 
    monoidal $1$-category $h_1\cC$. The $1$-categorical definition
    was given in the introduction.
	If all objects of $\cC$ have duals we say that $\cC$ has duals.
    Define $\cC^\fd$ to be the maximal (full) subcategory of $\cC$ 
    that has duals.
\end{defn}

\subsection{Some functors of interest} 

\begin{defn}\label{defn:functors}
	For a complete Segal space $X \in \CSS$ we define 
	its \emph{space of objects} as $\obj(X) := X_0$.
	Its \emph{space of endomorphisms} is defined as 
	$\mcE(X) := X_0 \times_{X_0\times X_0} X_1$, 
	where the two maps are the diagonal $\gD:X_0 \to X_0 \times X_0$
	and the source-target map $(d_0,d_1):X_1 \to X_0 \times X_0$.
    The \emph{space of automorphisms} of $X$ is defined as 
    $\mcA(X) := X_0 \times_{X_0\times X_0} X_1^{eq}$.
    Here $X_1^{eq} \subset X_1$ is the subspace of those connected components 
    that represent invertible morphisms in the homotopy category.
\end{defn}

\begin{defn}\label{defn:functors-symmon}
    For a special $\gC$-object in complete Segal spaces $X \in \CSS^\ot$ 
    we let $\obj^\fd(X) \subset \obj(X) = X_0^{\gle{1}}$ denote the union of those
    connected components that correspond to dualisable objects in the homotopy category.
    Accordingly, we let $\mcE^\fd(X) \subset \mcE(X)$ and $\mcA^\fd(X) \subset \mcA(X)$
    be the subspaces supported at the dualisable objects.
	Finally, we set
	\[
		\Sc(X) := X_0^{\gle{0}} 
        \times_{(X_0^{\gle{1}} \times X_0^{\gle{1}})}  X_1^{\gle{1}}.
	\]
	Note that this agrees with the functor $\End_{\blank}(\1)$ defined above 
	\cite[Proposition 2.7]{TV15}.
\end{defn}

\begin{remark}
    By construction there is a natural transformation $\mcA \to \mcE$ 
    such that for each $X$ the map $\mcA(X) \to \mcE(X)$ is an equivalence
    onto the connected components it hits.
\end{remark}

\begin{lem}\label{lem:A=L}
    For any complete Segal space represented by a simplicial topological space $X_\cd$ 
    we can compute $\mcE(X)$ as the space of tuples $(\gc, f)$ where 
    $f \in X_1$ and $\gc$ is a path from $d_0f$ to $d_1f$.
    The space $\mcA(X)$ is the subspace of those $(\gc,f)$ where $f \in X_1^{eq}$,
    i.e.\ $f$ is invertible in the homotopy category. 
    Moreover, $\mcA(X)$ is equivalent to the free loop space $\gL(\obj(X)) = \Map(S^1, X_0)$.
\end{lem}
\begin{proof}
    In order to compute the homotopy pullback 
    $\mcA(X) = X_0 \times_{X_0\times X_0}^h X_1^{eq} $
    we replace the diagonal $X_0 \to X_0 \times X_0$ by the path fibration
    $P(X_0) \to X_0 \times X_0$ and obtain:
    \[
        \mcE(X) \simeq P(X_0) \times_{X_0\times X_0} X_1.
    \]
    Note that the right-hand space is indeed the space of tuples $(\gc,f)$
    as described in the lemma. 
    For the second claim, recall from remark \ref{rem:Rezk-complete} 
    that for a complete Segal space the degeneracy
    map $s_0:X_0 \to X_1$ induces an equivalence $X_0 \simeq X_1^{eq}$.
    We therefore have an further equivalence 
    \[
        \mcA(X) \simeq P(X_0) \times_{X_0\times X_0} X_1^{eq}
        \simeq P(X_0) \times_{X_0\times X_0} X_0.
    \]
    The latter space is the space of paths in $X_0$ whose
    start and end-point agree, i.e.\ the space of free loops $\gL(X_0) = \Map(S^1, X_0)$.
\end{proof}

\begin{lem}\label{lem:functors}
	Write $X \in \CSS^\ot$ as $X\simeq N\cC$ for some symmetric monoidal 
    \icat\ $\cC$.
	Then the above functors can be described as follows:
	\begin{itemize}
		\item $\obj(X)$ is the maximal subgroupoid $\cC^\sim$ of $\cC$ 
            interpreted as a space,
		\item $\obj^{\fd}(X)$ is the $\infty$-groupoid $(\cC^\fd)^\sim$ 
            interpreted as a space,
		\item $\Sc(X)$ is the endomorphism space $\hom_\cC(\1,\1)$ of the 
            unit object $\1 \in \cC$.
        \item $\mcA^\fd(X)$ is equivalent to $\Map(S^1, \obj^\fd(X))$,
        which is the functor $\mcL\mc{I}$ in \cite[498]{TV15}.
	\end{itemize}
\end{lem}
\begin{proof}
	For $\obj(X)$ this is a direct consequence of the definition on 
    $N\cC$, as discussed in \ref{nrem:recover}.
	Since the notion of dualisability is defined via 
    the homotopy-category the description of $\obj^\fd$ follows immediately.
	
	The definition of $\Sc(X)$ is somewhat cryptic, but in fact not much 
    is happening:
	the space $X_0^{\gle{0}}$ is contractible since the $\gC$-object is 
    special.
	We may hence replace it by the terminal space $*$.
	The functor $* \to X_0^{\gle{0}} \to X_0^{\gle{1}} = \cC^\sim$ picks 
    out the unit object $\1$ of $\cC$ and so the definition is 
    equivalent to 
	\[
		\Sc(X) \simeq \{\1\} \times_{X_0^{\gle{1}} }  X_1^{\gle{1}} 
        \times_{X_0^{\gle{1}} } \{\1\} 
				= \{\1\} \times_{\cC^\sim}  
                (\cC^{[1]})^\sim \times_{\cC^\sim} \{\1\}  \simeq 
                \hom_\cC(\1,\1)
	\]
    as claimed.
    
    Finally, the claim that $\mcA^\fd(X) \simeq \Map(S^1, \obj^\fd(X))$ 
    follows by the same argument as in lemma \ref{lem:A=L},
    now restricted to the full subcategory on dualisable objects.
\end{proof}

We also show that these definitions are indeed compatible with the one-categorical
definitions given in the introduction:
\begin{lem}\label{lem:recov}
    For $\cC\in \CSS^\ot$ there are natural bijections
    $\pi_0 \Sc(\cC) \cong \Sc(h_1 \cC)$,
    $\pi_0 \mcE^\fd(\cC) \cong E^\fd(h_1\cC)$, and 
    $\pi_0 \mcA^\fd(\cC) \cong A^\fd(h_1\cC)$.
\end{lem} 
\begin{proof}
    In lemma \ref{lem:functors} we provided a natural equivalence $\Sc(\cC) \simeq \hom_\cC(\1,\1)$. After applying $\pi_0$ this becomes
    \[
        \pi_0 \Sc(\cC) \cong \pi_0 \hom_\cC(\1,\1) = \Hom_{h_1\cC}(\1,\1) = \Sc(h_1 \cC).
    \]
    For the second part let us assume that every object in $\cC$ is dualisable
    so that $\mcE^\fd(\cC) = \mcE(\cC)$. Since duals are computed on the level
    of homotopy categories this will not cause a problem.
    
    The homotopy fibre of the map $\mcE(\cC) \to N_0\cC$ at an object $x \in \cC$
    is equivalent to $\hom_\cC(x,x)$, so the connected components of the fibre are 
    $\pi_0\hom_\cC(x,x) = \End_{h_1\cC}(x)$. The fundamental group of $N_0\cC$ at $x$
    is $\Aut_{h_1\cC}(x)$. Therefore, choosing representatives $x_i$ for the 
    isomorphism classes in $\cC$, we can write $\pi_0\mcE(\cC)$ as the disjoint
    union of quotients: $\coprod_{i \in \pi_0 N_0\cC} \End_{h_1\cC}(x)/\Aut_{h_1\cC}(x)$.
    Here the action is by conjugation and the resulting set 
    is canonically isomorphic to $E(h_1\cC)$.
    This isomorphism restricts to a bijection between $\mcA(\cC)$ and $A(h_1\cC)$.
\end{proof}


\section{The bordism category and the cobordism hypothesis}
\label{sec:Bord}

In order to apply the cobordism hypothesis, which is a key step in the proof of the main theorem,
we need to first establish a concrete model for the one dimensional bordism category.
In this section we take the model for $\Born{1}(X)$ from \cite{SP17}
and show that it is a special $\gC$-object in complete Segal spaces.
Moreover, we define a ``marked" version of the bordism category 
and show that it too is a special $\gC$-object in complete Segal spaces.
Using this model we can then formulate a special case of the conjectural 
cobordism hypothesis with singularities.

\subsection{The bordism category as a complete Segal space}

    We are going to define a symmetric monoidal \icat\ $\Born{d}(X)$
    by constructing a $1$-functor $F:\gD^{op} \times \gC^{op} \to \Top$ based on \cite{SP17}
    and then showing in \ref{thm:Bord-Segal} that after composing with 
    $\Top \to \Spc$ it satisfies the required conditions 
    to be in $\CSS^\ot$ as in theorem \ref{thm:smCSS}.
    
    In what follows we let $\IR^\infty$ denote the countably dimensional vector space
    $\IR^\infty = \colim_{n \to \infty} \IR^n$ and
    we let $\IR^{1+\infty}$ denote the product $\IR \times \IR^\infty$.

\begin{defn}[\!{\cite[Definition 5.7]{SP17}}]
	For a submanifold $W \subset \IR^{1+\infty}$ and an interval $U \subset \IR$ 
	we write $W_U := W \cap (U \times \IR^\infty)$.
	$W$ is called \emph{cylindrical} over $U$ if there is 
	a closed $(d-1)$-dimensional submanifold $N \subset \IR^\infty$ 
	such that $W_U = U \times N$.
	
	For $[n] \in \gD$ we let $\IR^{[n]}$ be the topological space of monotone maps $[n] \to \IR$. 
	We say that a submanifold $W \subset \IR^{1+\infty}$ is admissible 
	with respect to $t \in \IR^{[n]}$ if $W$ is closed as a subset,
	the projection $\pi:W \to \IR$ to the first coordinate is proper,
	and there is an $\eps>0$ such that $W$ is cylindrical over each of the intervals 
	$(t_i-\eps, t_i+ \eps)$ for $i=0,\dots, n$.
\end{defn}

To topologise the space of bordisms we recall the plot-topologies from \cite{SP17}.
\begin{defn}\label{defn:space-of-submfds}
    For a space $X$ we let $\Psi_d(X)$ denote the set of all tuples $(W, \gp)$ 
    where $W \subset \IR^\infty$ is an oriented $d$-dimensional submanifold 
    that is closed as a subset and $\gp:W \to X$ is a continuous map.
    
    For any $k$-dimensional manifold $U$ we say that a map $f: U \to \Psi_d(X)$
    is \emph{smooth} if the graph
    \[
        \gC(f) = \{ (u,v) \in U \times \IR^\infty \;|\; v \in f(u)\}
    \]
    is a smooth submanifold of $U \times \IR^\infty$ and the map $\gp_f:\gC(f) \to X$
    is continuous.
    The \emph{plot topology} on $\Psi_d(X)$ is the finest topology such that all
    smooth maps $U \to \Psi_d(X)$ are continuous.
\end{defn}

We are now ready to define the simplicial space that will give rise to the bordism category.
Our model is almost identical to the one defined in \cite[Definition 5.8]{SP17}
with the only difference being that we chose to encode the symmetric monoidal
structure by using $\gC$-spaces whereas Schommer-Pries uses $E_\infty$-algebras.

\begin{defn}\label{defn:PBord}
    For any topological space $X$, we define $\Born{d}(X)$ by the functor 
    $\gD^{op} \times \gC^{op} \to \Top$ 
    that sends $([n], \gle{k})$ to the topological space of tuples 
    $
        (t, (W_1,\gp_1), \dots, (W_k,\gp_k))
    $
    where $t\in \IR^{[n]}$ and the $W_i$ are pairwise disjoint 
    $d$-dimensional oriented submanifolds of $\IR\times\cube$ 
    admissible with respect to $t$.
    The $\gp_i$ are continuous maps $\gp_i: W_i \to X$.
    This is topologised %
    as a subspace of $\IR^{[n]} \times (\Psi_d(X))^k$. 
    Functoriality in the $\gC$-direction is defined as follows.
    For $\gl: \gle{k} \to \gle{l}$ we set 
    \[
        \gl_* (t, (W_1,\gp_1), \dots, (W_k,\gp_k)) 
        = \left( t, \coprod_{\gl(i_1) =1} ( W_{i_1}, \gp_{i_1}), \dots, \coprod_{\gl(i_l) =l} (W_{i_l},  \gp_{i_l})\right).
    \]
    In the $\gD$-direction a morphism $\rho: [n] \to [m]$ 
    acts by reindexing the $t_i$: 
    $(\rho^*t)_j = t_{\rho(j)}$.
\end{defn}

We now proceed to show that the above satisfies the Segal and completeness condition,
so that it lies in $\CSS^\ot \subset \Fun(\gD^{op}\times \gC^{op}, \Spc)$
and gives rise to a symmetric monoidal infinity category via theorem \ref{thm:smCSS}.
\begin{thm}\label{thm:Bord-Segal}
    For any $d \ge 0$ and $X \in \Top$ the simplicial space $\Born{d}(X)^{\gle{1}}_\cd$
    is a Segal space and the functor $\Born{d}(X):\gD^{op} \times \gC^{op} \to \Spc$
    is a special $\gC$-object in Segal spaces.
    If $d \le 2$, then $\Born{d}(X)$ is also complete and 
    defines a special $\gC$-object in complete Segal spaces.
\end{thm}
\begin{proof}
    We begin with the Segal condition for $\Born{d}(X)^{\gle{1}}$. 
    Let $B_n = \{(t\in \IR^{[n]}, (W,\gp)) \;|\; \dots \}$
    be the simplicial topological space
    from definition \ref{defn:PBord} that represents $\Born{d}(X)^{\gle{1}}$.
    
    For any $n$ we let $B_n' \subset B_n$ be the subspace of those $(t, (W,\gp))$
    satisfying that $t_i = i$, that $W$ is cylindrical over $(-\infty,0]$ and $[n,\infty)$,
    and that the two restrictions $\gp: W_{(-\infty,0]} \to X$ and
    $\gp: W_{[n,\infty)} \to X$ factor through the projection to $W_{\{0\}}$ 
    and $W_{\{n\}}$ respectively.
    These conditions ensure that $(t, (W,\gp))$ is uniquely determined by the restriction
    $(W_{[0,n]}, \gp_{|W_{[0,n]}})$.
    This encodes exactly the data of $n$ composable bordisms of length $1$.
    Put into formulas we have a homeomorphism:
    \[
        B_n' \xrightarrow{\quad \cong \quad} B_1' \times_{B_0'} \dots \times_{B_0'} B_1'
    \]
    defined by sending $(W,\gp)$ to the tuple 
    $(W_{[i,i+1]}, \gp_{|W_{[i,i+1]}})_{i=0}^{n-1}$.
    To be precise we should actually translate $W_{[i,i+1]}$ by $i$ to the left
    and then extend it cylindrically over $(-\infty,0]$ and $[1,\infty)$,
    then it is a well-defined point in $B_1' \subset B_1$.
    
    The above homeomorphism already indicates that $B_n'$ satisfies some type of Segal condition.
    However, one should note that $B_n' \subset B_n$ is not closed under face 
    or degeneracy operators, so the $B_n'$ do not actually define a simplicial space.
    Still, this will be very useful as we have the following homotopy commutative diagram:
    \[
    \begin{tikzcd}
        B_n' \ar[d, "\cong"] \ar[rr, "i_n"] 
        && B_n \ar[d, "S"] 
        \\
        B_1' \times_{B_0'} \dots \times_{B_0'} B_1' \ar[r, "q"] 
        & B_1' \times_{B_0'}^h \dots \times_{B_0'}^h B_1' \ar[r, "j"] 
        & B_1 \times_{B_0}^h \dots \times_{B_0}^h B_1
    \end{tikzcd}
    \]
    Our goal is to show that the map $S$ is an equivalence -- this is the Segal condition.
    The map $i_n$ is the canonical inclusion map $i_n: B_n' \to B_n$.
    It has a homotopy inverse $B_n \to B_n'$ defined by piece-wise linearly 
    rescaling the first coordinate direction of $\IR \times \cube$
    to that $t_i = i$ and then pushing off everything outside of $[0,n]$ to infinity.
    This is a standard type of argument and we refer the reader 
    to \cite[Proof of 3.9]{GRW10} for details.
    This also implies that $j$ is an equivalence as it is a homotopy pullback of $i_1$ 
    and $i_0$.
    
    It remains to check that the map $q$ that compares the strict pullback
    with the homotopy pullback is an equivalence. 
    For this it will suffice to show that the maps $B_1' \rightrightarrows B_0'$
    involved in the pullback are Serre fibrations.
    But now observe that $B_1'$ is in fact homeomorphic to the space 
    $\coprod_{[W]} \mathcal{M}^X(W)$ discussed in the start of \cite[section 3.1]{ERW19}.
    This is true because the bordism category from \cite{ERW19} agrees
    with the one from \cite{GRW10} after removing units 
    and \cite[Theorem A.2]{SP17} shows that their topology in turn is homeomorphic
    to the plot topology we used.
    We can therefore cite \cite[Proposition 3.2.4(ii)]{ERW19},
    which tells us that the source and target maps 
    $s,t:B_1' \to B_0'$ are Serre fibrations.
    This shows that $q$ is an equivalence and hence the map $S$ also has to be,
    and $B_\cd$ is a Segal space for any $d$ and $X$ as claimed.

    The next claim is that $B_\cd$ is complete for $d \le 2$.
    As in remark \ref{rem:Rezk-complete} we let $B_1^{\rm eq} \subset B_1$ denote the subspace
    of those $1$-simplices that represent invertible morphisms in the homotopy category
    $h(B_\cd)$. This is always a union of connected components.
    By \cite[Theorem 6.2]{Rez01} $B_\cd$ is complete if and only if
    $d_0: B_1^{\rm eq} \to B_1$ is an equivalence.
    To prove this, we use that $h(B_\cd)$ is equivalent to $\Cob_d(X)$ 
    by lemma \ref{lem:hBord=Cob}. Since $d \le 2$ it is not difficult to see that
    if a bordism $W:M \to N$ is invertible in $\Cob_d(X)$, 
    then it is diffeomorphic to $M \times [0,1]$.
    (See remark \ref{rem:completeness} for why this does not work for general $d$.)
    As before we use \cite[section 3.1]{ERW19} to see that 
    $B_1^{eq} \subset B_1$ is equivalent to 
    $\coprod_{[M]} \Map(M\times [0,1], X) /\!\!/ \Diff^+(M \times [0,1])$ and
    the space of objects $B_0$ is equivalent to $\coprod_{[M]} \Map(M,X)/\!\!/\Diff^+(M)$.
    In both cases the coproduct runs over representatives for diffeomorphism
    classes of closed oriented $(d-1)$-manifolds.
    The map $B_0 \to B_1^{eq}$ is an equivalence for $d \le 2$ 
    because $\Diff^+(M\times [0,1]) \simeq \Diff^+(M)$
    holds for all $(\le 1)$-manifolds $M$, but it is not true for general $d$
    as the pseudo-isotopy space is generally non-trivial.
    
    Finally, it remains to show that $\Born{d}(X)$ satisfies the specialness 
    condition as a $\gC$-object in simplicial spaces.
    This can be checked in each level separately, and relies on the idea
    that $\cube$ is ``large enough" for us to make any two submanifolds
    disjoint, canonically up to a contractible space of choices.
    We refer the reader to \cite[Proposition 7.2]{CS19} for details.
\end{proof}

\begin{remark}
    The first model for a bordism category satisfying the Segal condition was constructed in \cite{Sch14}.
    There the completeness condition is enforced by replacing $\Born{d}$ with its completion.
    This is not necessary for $d \le 2$ as the more naive construction is already complete 
    -- see \cite[Remark 5.25]{CS19}.
    Since we have not been able to find a proof for this claim in the literature, we give the one above.
\end{remark}

\begin{remark}\label{rem:completeness}
    Note that for $d>2$ the simplicial space $\Born{d}(X)^{\gle{1}}$ is still a Segal space,
    but the completeness condition is not automatic.
    In fact, we used in the proof that every invertible bordism $W$ is of the form $M \times [0,1]$.
    For general dimension $d$ two objects $M, N \in \Cob_d$
    are isomorphic if and only if $M \times \IR$ and $N \times \IR$
    are diffeomorphic, see \cite[Proposition 3.3]{HJ18}.
    When $d \neq 4$ an isomorphism $M \to N$ in $\Cob_d$ is exactly an h-cobordism
    \cite[Proposition 3.11]{HJ18}. 
    Therefore when $d\ge 5$ the existence of non-trivial $h$-cobordism 
    implies that $B_\cd$ is not complete.
\end{remark}

We now show that $\Born{d}(X)$ indeed recovers the desired $1$-category $\Cob_d(X)$
as its homotopy category. The lemma is based on \cite[proposition 8.20]{CS19},
which shows the same in a different model for $\Born{d}(X)$.
\begin{defn}\label{defn:Cob}
    The $1$-category $\Cob_d(X)$ has as objects tuples $(M,\gp)$ 
    where $M$ is a closed oriented $(d-1)$-manifold 
    and $\gp:M \to X$ is a continuous map.
    A morphism $(M,\gp) \to (N,\psi)$ is represented by a triple $(W, i, \chi)$
    where $W$ is a compact oriented $d$-manifold, $i:M^- \amalg N \cong \partial W$
    is an orientation-preserving diffeomorphism, and $\chi:W \to X$
    is a continuous map such that $\chi \circ i = \gp \amalg \psi$.
    Two such triples $(W,i,\chi)$ and $(W',i',\chi')$ represent the same morphism
    if there is an orientation preserving diffeomorphism $f: W \cong W'$
    such that $f_{|\partial W} \circ i = i'$ and that $\chi' \circ f$ is homotopic
    to $\chi$ relative to $\partial W$.
    Composition is defined by gluing cobordisms.
    This category has a symmetric monoidal structure induced by disjoint union.
\end{defn}
\begin{lem}\label{lem:hBord=Cob}
    For all $X$ the homotopy category $h_1 \Born{d}(X)$ is canonically equivalent
    to the symmetric monoidal category $\Cob_d(X)$ described in definition \ref{defn:Cob}.
\end{lem}
\begin{proof}
    As discussed in \ref{nrem:h-cat} we are free to choose representatives
    for $\pi_0 \Born{d}(X)_0$. In fact choosing multiple representatives
    per connected component yields an equivalent category, so we can 
    simply set the objects of $h_1 \Born{d}(X)$ to be tuples 
    $(\IR \times M, \gp \circ \pr_M)$
    where $M \subset \cube$ is a closed oriented $(d-1)$-manifold
    and $\gp: M \to X$ is a continuous map.
    This maps to the objects of $\Cob_d(X)$ by forgetting the embedding.
    
    The set of morphisms $(M,\gp) \to (N,\psi)$ can be computed
    as $\pi_0$ of the homotopy pullback 
    \[
        \{(M,\gp)\} \times_{\Born{d}(X)_0}^h \Born{d}(X)_1 \times_{\Born{d}(X)_0}^h \{(N,\psi)\}.
    \]
    Up to equivalence we may replace the middle space with the space 
    $B_1'$ of bordism of length $1$ from the proof of theorem \ref{thm:Bord-Segal}.
    So we are interested in the homotopy fibre of the map
    $B_1' \to (B_0')^2$ sending $W$ to $(W_0, W_1)$.
    By \cite[Proposition 3.2.4(ii)]{ERW19} this map is a fibration,
    so we can consider the strict fibre instead.
    From the description in \cite[section 3.1]{ERW19} it follows 
    that this is given by:
    \[
        \Hom_{h_0\Born{d}(X)}((M,\gp), (N,\psi))
        = \coprod_{[W, \partial W = M^- \amalg N]} 
        \pi_0 \Map(W_{[0,1]},X; \gp \amalg \psi)/\Diff^+(W).
    \]
    Here $W$ only runs over equivalence classes of abstract bordisms 
    that go from $M$ to $N$
    and the maps $W \to X$ are required to restrict to $\gp$ and $\psi$,
    respectively.
    There is a canonical functor $h_1 \Born{d}(X) \to \Cob_d(X)$ 
    defined by forgetting the embedding on objects and sending
    bordisms to their equivalence class.
    It follows from the above that this functor is fully faithful.
    Moreover, it is essentially surjective because every closed
    oriented $(d-1)$-manifold can be embedded into $\cube$.
    
    Moreover, if we define a $\gC$-structure on $\Cob_d(X)$ 
    the same way we did on $\Born{d}(X)$, then this is an equivalence 
    of $\gC$-categories, and the $\gC$-structure on $\Cob_d(X)$
    corresponds to the standard symmetric monoidal structure 
    defined by disjoint union.
    We refer the reader to the proof of \cite[proposition 8.20]{CS19}
    for more details on how this is compatible with the $\gC$-structure.
\end{proof}

\subsection{The cobordism hypothesis}
\begin{lem}\label{lem:BordHasDuals}
    For all $X$ the symmetric monodial $\infty$-category $\Born{d}(X)$  has duals.
\end{lem}
\begin{proof}
    By definition \ref{defn:dual} we need to show that every object in the homotopy category
    $h_1 \Born{d}(X)$ is dualisable. Because of lemma \ref{lem:hBord=Cob}
    we can equivalently construct duals for the category $\Cob_d(X)$.
    This is well-known to those familiar with the cobordism hypothesis,
    but we include a brief description in the interest of completeness.
    
    Let $(M, \gp)$ be some object of $\Cob_d(M)$. Then we will show that 
    $(M^-, \gp)$ is the dual. 
    As evaluation $\ev:(M^- \amalg M, \gp \amalg \gp) \to (\emptyset, \emptyset)$ 
    we take the manifold $W = M \times [0,1]$ with the boundary identification
    $\partial W \cong (M^- \amalg M)^- \amalg \emptyset$.
    This will be equipped with the map $\gp \circ \pr_M: M \times [0,1] \to M \to X$.
    Similarly, the coevaluation 
    $\co:(\emptyset, \emptyset) \to (M \amalg M^-, \gp \amalg \gp)$
    is $W$, but with the boundary identification 
    $\partial W \cong \emptyset \amalg (M \amalg M^-)$.
    The two compositions $(\id_M \amalg \ev) \circ (\co \amalg \id_M)$
    and $(\ev \amalg \id_{M^-}) \circ (\id_{M^-} \amalg \co)$ 
    are equivalent to the identity bordisms $M \times [0,1]$ and
    $M^- \times [0,1]$, respectively.
    This is best checked by drawing out the bordisms 
    as illustrated in figure \ref{fig:Zorro}.
\end{proof}

\begin{figure}[ht]
    \centering
    \footnotesize
    \def\svgwidth{.9\linewidth}
    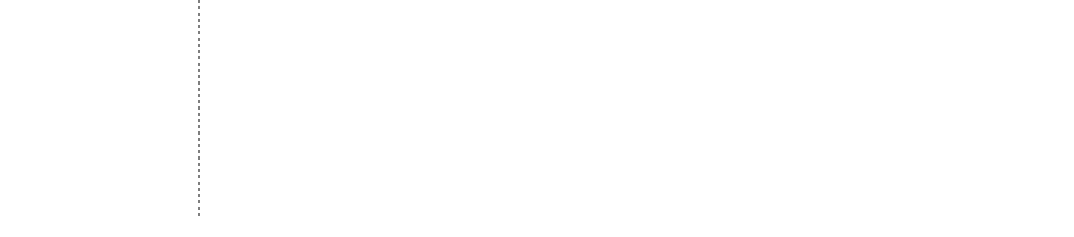
    \caption{This graphic verifies the relation 
    $(\ev \amalg \id_M) \circ (\id_M \amalg \co) = \id_M$
    in the homotopy category $h\Born{d}$ for the case $d=1$
    and $M=*_+$.
    The general case can be shown by taking the product of
    the above with an arbitrary manifold $M$.}
    \label{fig:Zorro}
\end{figure}

\begin{remark}
    One can also equip bordism categories with more general tangential structures.
    In the context of the cobordism 
    hypothesis one usually uses the framed bordism category $\Bfrn{d}$
    where each $W$ comes with a framing.
    There always is a functor $\Bfrn{d}(X) \to \Born{d}(X)$ 
    that forgets the framing and only remembers the induced orientation.
    For $d=1$ the space of framings on a fixed $W$ that are compatible with a preferred orientation
    is contractible. Therefore $\Bfrn{d}(X) \to \Born{d}(X)$ is a level-wise equivalence of 
    Segal spaces and the associated symmetric monoidal \icats\ are equivalent.
    As we will only need the one-dimensional cobordism hypothesis
    we can therefore state it using the oriented bordism category $\Born{1}(X)$
    rather than $\Bfr(X)$.
\end{remark}

\begin{nrem}
    To formulate the cobordism hypothesis in dimension $1$ we need to construct
    a natural map $X \to \obj^\fd(\Bor(X))$ for any space $X$.
    As we saw above all objects in $\Bor(X)$ are dualisable, and so the space 
    of dualisable objects is simply $\Bor(X)_0$.
    We can define the map $f: X \to \Bor(X)_0$ by sending $x \in X$ to the 
    (positively oriented) manifold $W :=\IR\times\{0\} \subset \IR\times \cube$
    equipped with the constant map $\gp_x: W \to \{x\} \to X$.
    This is natural in $X$ and induces a map:
    \[
        \hom_{\Catit}(\Bor(X), \cC) \xrightarrow{\obj^\fd} \hom_{\Spc}(\obj^\fd(\Bor(X)), \obj^\fd(\cC)) \xrightarrow{f^*} \hom_{\Spc}(X,\obj^\fd(\cC)).
    \]
\end{nrem}

\begin{thm}[Cobordism hypothesis in dimension $1$,{ \cite{LurCH}, for 
    more details see \cite{Har12}}]\label{thm:cobhyp}
    The map constructed above is an equivalence for all $X\in \Top$ and $\cC \in\Catit$:
    \[
        \hom_{\Catit}(\Bor(X), \cC) \simeq \hom_{\Spc}(X,\obj^\fd(\cC)).
    \]
\end{thm}

We will be particularly interested in the case $X = S^1$.
We now choose the positively oriented point 
$*_+ \in \Bor(S^1)$ and a preferred endomorphism $\ga: *_+ \to *_+$
as follows:
\begin{defn}
    The object $*_+$ is defined as
    $(t=0, \IR \times \{0\}, \gp) \in (\Bor(S^1))_0^{\gle{1}}$
    where $\IR$ is given the standard orientation 
    and $\gp:\IR \to S^1$ is the constant map that sends everything to the base point
    $0 \in \IR/\IZ =: S^1$.
    The endomorphism  $\ga: *_+ \to *_+$ 
    has as underlying bordism is the trivial bordism 
    $((t_0=0, t_1=1), W = \IR \times \{0\})$ of length $1$.
    We equip this with the labelling $\psi: \IR \to S^1 = \IR/\IZ$ 
    defined as $\psi(x) = [x]$ for $x \in [0,1]$ and $\psi(x) = [0]$ otherwise.
    
    Since $\ga$ is represents an invertible morphism in the homotopy category
    this defines a point $(*_+, \ga) \in \mcA(\Bor)$ in the space of automorphisms.
\end{defn}

\begin{cor}\label{cor:cobhyp-A}
    For every $Y \in \CSS^\ot$ evaluating on 
    $(*_+,\ga)\in \mcA^\fd(\Bor(S^1))$ yields an equivalence
    \[
        \hom_{\CSS^\ot}(\Bor(S^1), Y) \simeq \mcA^\fd(Y).
    \]
\end{cor}
\begin{proof}
    Consider the full subcategory $Y^\fd \subset Y$ on the dualisable objects.
    By the cobordism hypothesis \ref{thm:cobhyp} and lemma \ref{lem:functors}
    we have equivalences
    \[
        \hom_{\CSS^\ot}(\Bor(S^1), Y) \simeq \hom_{\Spc}(S^1, (Y^\fd)_0^{\gle{1}})
        \simeq \mcA^\fd(Y).
    \]
    This equivalence corresponds to the evaluation at the point of $\mcA(\Bor(S^1))$,
    which is represented by the free loop $\gc: S^1 \to \obj(\Bor(S^1))$
    sending $x$ to the manifold $*_+$ labeled by $l:* \to \{x\} \subset S^1$.
    Following through the equivalence in the proof of lemma \ref{lem:A=L}
    we see that this is in the same connected component as the preferred endomorphism
    $(\ga:*_+ \to *_+) \in \mcA(\Bor(S^1))$.
    Therefore the natural transformation described in the claim 
    is an equivalence as well.
\end{proof}

\subsection{The cobordism category with marked points}
We now introduce a variant of the bordism category where morphisms 
$W: M \to N$ come equipped with a finite subset 
$A \subset W\setminus \partial W$ of marked points. 
Then we show that it is a complete Segal space and use it to formulate
a special case of Lurie's cobordism hypothesis with singularities in our setting.

\begin{defn}\label{defn:marked-submfd}
    For $d\ge 0$ we let $\Psi_d^m$ denote the 
    \emph{space of oriented marked $d$-dimensional submanifolds}  of $\IR^\infty$:
    a point in $\Psi_d^m$ is a tuple $(W,A)$ 
    where $W \subset \IR^{1+\infty}$ is a $d$-dimensional oriented submanifold
    and $A \subset W$ is a $0$-dimensional submanifold.
    We topologise this with a plot topology as in definition  \ref{defn:space-of-submfds}
    except that now both the graph obtained from $W$ 
    and the graph obtained from $A$ have to be smooth.
\end{defn}

    Note that $\Psi^m$ can also be thought as the subspace of $\Psi_d \times \Psi_0$ 
    containing precisely those tuples $(W,A)$ where $A \subset W$.

\begin{defn}
    We define a simplicial $\gC$-space $\Bormn{d} \in \sP(\gD\times \gC)$ by the functor 
    $\gD^{op} \times \gC^{op} \to \Top$ 
    that sends $([n], \gle{k})$ to the topological space of tuples%
    \[
        (t, (W_1,A_1), \dots, (W_k,A_k))
    \]
    where $t\in \IR^{[n]}$ and the $W_i$ are pairwise disjoint 
    $d$-dimensional oriented submanifolds of $\IR \times \cube$,
    each of which is admissible with respect to $t$.
    The $A_i$ are finite subsets of $W_i$ such that 
    $\pi(A_i) \subset \bigcup_{j=1}^n (t_{j-1}, t_j)$.
    This is topologised as a subspace of $\IR^{[n]} \times (\Psi^m)^k$.
    Functoriality in $([n], \gle{k})$ is as in definition \ref{defn:PBord},
    except that we also need to forget all points of $A_i$ that 
    lie outside of $(t_0,t_n)$ after applying a face operator.
\end{defn}

We have an analogue of \ref{thm:Bord-Segal}:
\begin{proposition}\label{prop:Bordm-CSS}
    For any $d\ge0$ the simplicial space $(\Bormn{d})^{\gle{1}}$
    is a Segal space and the functor 
    $\Bormn{d}:\gD^{op} \times \gC^{op} \to \Spc$
    is a special $\gC$-object in Segal spaces.
    When $d\le 2$ this Segal space is moreover complete.
\end{proposition}
\begin{proof}
    We will show how the different parts of the proof of theorem \ref{thm:Bord-Segal}
    generalise to the marked case.
    Let us write $B_\cd := (\Born{d})^{\gle{1}}_\cd$ 
    and $B_\cd^m := (\Bormn{d})^{\gle{1}}_\cd$.
    We let $B_n' \subset B_n$ and $B_n^{m\prime} \subset B_n^m$ 
    be the length $1$ versions as in the proof of theorem \ref{thm:Bord-Segal}.
    
    Our proof of the Segal property relied on the fact that the maps
    $s,t: B_1' \to B_0'$ that send a bordism of length one to its source
    or target are fibrations. 
    In the next paragraph we will argue that 
    the map $p: B_1^{m\prime} \to B_1'$ that forgets the markings is 
    a locally trivial fiber bundle whose fiber at a point $W \in B_1'$
    is the unordered configuration space $\Conf_*(W_{(0,1)})$.
    In particular $p$ is a Serre fibration.
    Since $B_0 = B_0^m$ it follows that the source and target maps of $B_\cd^m$
    are composites of Serre fibrations $B_1^{m\prime} \to B_1' \to B_0 = B_0^m$.
    Therefore the same proof as in theorem \ref{thm:Bord-Segal} applies
    and $B_\cd^m$ is a Segal space.
    
    To see that $p: B_1^{m\prime} \to B_1'$ is a fiber bundle let $W \in B_1'$
    any point in the base. We choose a tubular neighbourhood 
    $N \subset \IR\times \cube$ of $W$ and an identification of $N$ 
    with the normal bundle $\nu_W$.
    For any smooth section of the normal bundle $f \in \gC(\nu_W)$ 
    the image $f(W) \subset N \subset \IR \times \cube$ is a smooth submanifold
    and for certain admissible $f$ it is an element $f(W) \in B_1'$:
    let $U_W \subset \gC(\nu_W)$ denote the subset of such $f$.
    This is in bijection with the set $U_W' \subset B_1'$ of those 
    $V \in B_1'$ such that there is an $f \in \gC(\nu_W)$ for which $f(W) = V$.
    It follows from \cite[Lemma A.5]{CS19} 
    that $U_W$ and $U_W'$ are homeomorphic and that
    $U_W'$ is a neighbourhood of $W \in B_1'$.
    (For this it is important to note that any $V \in B_1'$ is uniquely determined
    by its intersection with the compact subspace $[0,1]^{1+\infty} \subset \IR^{1+\infty}$.)
    All that is left to do is to trivialise $p:B_1^{m\prime} \to B_1'$ over $U_W'$.
    Indeed we have the following homeomorphism:
    \[
        U_W \times \Conf_*(W_{(0,1)}) \cong p^{-1}(U_W) , \quad
        (f, A \subset W_{(0,1)}) \mapsto (f(W), f(A)). 
    \]

    Next, we check the completeness condition for $B_\cd^m$.
    For this it is useful to think of $B_\cd$ 
    as the subspace of $B_\cd^m$ given by the manifolds with empty configurations.
    From the definition of the topology we can see that we cannot change
    the number of points in a configuration by a continuous path
    and so $B_n \subset B_n^m$ is a union of connected components 
    for each layer $n$.
    Moreover, note that for a morphism to be homotopy invertible in $B_\cd^m$
    both it and its inverse have to be labeled by an empty configuration:
    this is true because composition adds the number of points in the configuration
    and the identity morphisms have empty configurations.
    Therefore, the spaces of equivalences agree $(B_1^m)^{eq} = B_1^{eq}$.
    Since we also have $B_0 = B_0^m$ this means that $B_\cd^m$
    is complete if and only if $B_\cd$ is.
    We have argued why this is true for $d \le 2$ in \ref{thm:Bord-Segal}.

    Finally, the specialness condition for the $\gC$-direction is again standard.
\end{proof}

We pick the positively oriented point $*_+$ as a preferred object in $\Bormn{1}$
and construct an endomorphism $\gb: *_+ \to *_+$.
\begin{defn}
    The endomorphism $\gb = (t,W,A):*_+ \to *_+$ is defined by 
    $(t_0,t_1)=(0,1)$, $W = \IR \times \{0\}$, and 
    $A = \{\tfrac{1}{2}\} \times \{0\} \subset W$.
    This defines an element $(*_+, \gb) \in \mcE(\Borm)$ in the space of endomorphisms
    from definition \ref{defn:functors}.
\end{defn}

We will now construct a functor $L:\Borm \to \Bor(S^1)$ that sends a marked bordism
$(W,A):M \to N$ to $(W,l):M \to N$ where
the labeling $l:W \to S^1$ maps most of $W$ to the basepoint, 
except for a small neighbourhood of $A$ where it loops around
the circle $S^1$ once per point in $A$.

\begin{construction}
    To define the functor $L$ we first need to enhance the $\infty$-category
    $\Bormn{1}$ to contain the (contractible) data of disjoint small 
    $\eps$-balls around the marked points.
    For any oriented $1$-manifold $W \subset \IR^{1+\infty}$ 
    we can define a signed distance function
    $d_W:W \times W \to \IR \amalg \{\infty\}$ 
    by setting $d(x,y)=\pm l$ where $l$ is the length of the shortest path from $x$ to $y$
    and the sign depends on whether that path agrees with the orientation of $W$.
    Let $\c{B}$ be the simplicial $\gC$-space defined just like $\Bormn{1}$
    except that each $(W_i, A_i)$ comes with a function $\eps_i:A_i \to (0,\infty)$
    satisfying, for all $a \in A_i$:
    \[
        2\eps_i(a) < \min\{ |d_W(a,b)| \;|\; b \in A_i \setminus\{a\}\} \qand
        \eps_i(a) < \min\{ |\pi(a) - t_j| \;|\; j\in \{0,\dots, n\}\}.
    \]
    Because the space of possible choices for $\eps_i$ is convex,
    the forgetful map $\c{B} \to \Borm$ is a level-wise equivalence.
    As a result $\c{B}$ is a complete Segal space.
    
    We will now construct a functor $L:\c{B} \to \Bor(S^1)$.
    Let $(t, (M_1,A_1,\eps_1), \dots, (W_k, A_k, \eps_k))$ be a point in 
    $\c{B}_n^{\gle{k}}$. Then we define $l_i:W_i \to S^1 = \IR/\IZ$ as follows:
    \[
        l_i(x) = \begin{cases}
            \frac{1}{2} + \frac{d(a,x)}{2\eps_i(a)}
            & \text{ if there is } a \in A_i \text{ with } |d(a,x)| \le \eps_i(a)
           \\
            0 & \text{ otherwise.}
        \end{cases}
    \]
    One checks that this is continuous and loops around $S^1$ once 
    in every $\eps_i(a)$-ball.
\end{construction}

We can now state our interpretation of Lurie's cobordism hypothesis with singularities
in the one-dimensional case.
As explained in warning \ref{warning} some of our theorems are conditional 
on this conjecture.
\begin{conj}\label{conj:singcobhyp}
    For any $\infty$-category $\cC$ evaluating on the morphism $(W,A):*_+ \to *_+$
    yields an equivalence
    \[
        \Fun(\Borm, \cC) \simeq \mcE(\cC^\fd).
    \]
\end{conj}

\begin{remark}[Comparison to Lurie's conjecture]
    We will now informally derive our formulation from Lurie's more general 
    cobordism hypothesis with one type of singularity as stated in
    \cite[Proposition 4.3.1]{LurCH}.
    
    In the general setting one fixes a $d-1$-dimensional manifold $Y$
    and then allows bordisms to have singularities of the form of a cone on $Y$.
    We will only need this in dimension $1$ and for the case $Y=S^0$.
    Of course a cone on $S^0$ is diffeomorphic to $D^1$ 
    and so instead of introducing actual singularities in a bordism $W$
    it suffices to keep track of the cone points.
    This is the set of marked points $A \subset W$. 
    A small ball around each marked point $a \in A$ is then is a cone on $S^0$.
    The cobordism hypothesis with singularities says that there is an equivalence 
    between functors $\Borm \to \cC$ and functors $\mcZ: \Bor \to \cC$
    together with a choice of morphism $\ga: \1 \to \mcZ(S^0)$.
    
    By the cobordism hypothesis without singularities
    $\Funo{\Bor,\cC}$ is equivalent to $(\cC^\fd)^\simeq$
    via the evaluation on the positively oriented point.
    Let $x\in\cC$ denote $\mcZ(*_+)$. Then the value of $\mcZ$ on $S^0$ is 
    $\mcZ(S^0) \cong \mcZ(*_+ \amalg *_-) \cong \mcZ(*_+) \ot \mcZ(*_-) \cong x \ot x^\vee$.
    By duality $\hom(\1, x \ot x^\vee)$ is equivalent to $\hom(x, x)$,
    so instead of $\ga:\1 \to \mcZ(S^0) = x \ot x^\vee$ we may equivalently 
    choose $\ga^\#:x \to x$.
    
    In summary, the data of a symmetric monoidal functor $\Borm \to \cC$ 
    is equivalent to that of a dualisable object $x \in \cC^\fd$
    together with an endomorphism $\ga^\#:x \to x$.
    In other words, to a point $(x,\ga^\#) \in \mcE(\cC^\fd)$.
\end{remark}

\begin{nrem}\label{nrem:mcL-commutes}
    By construction the functor $L:\Borm \to \Bor(S^1)$ sends the preferred
    endomorphism $\gb:*_+ \to *_+$ in $\Borm$ to the preferred automorphism 
    $\ga: *_+ \to *_+$ in $\Bor(S^1)$.
    Therefore the following diagram commutes:
    \[
        \begin{tikzcd}
            \Funo{\Bor(S^1), \cC} \ar[r, "\simeq"] \ar[d, "L^*"] & \mcA(\cC^\fd) \ar[d]\\
            \Funo{\Borm, \cC} \ar[r, "\simeq"] & \mcE(\cC^\fd)
        \end{tikzcd}
    \]
    Moreover by the two variants of the cobordism hypothesis 
    the horizontal morphisms in this diagram are equivalences for all $\cC$.
    This means that via the Yoneda lemma the functor $L$ 
    corresponds to the natural transformation $\mcA^\fd \Rightarrow \mcE^\fd$.
\end{nrem}


\section[1-categorical classification]{$1$-categorical classification}
\label{sec:1class}
In this section we complete the proof of the classification of $1$-categorical tracelike transformations sketched in the introduction.

\begin{defn}
    The category $\CoZ$ has as objects finite sets $M$ equipped with an orientation $M \to \{+,-\}$.
    A morphism $X:M \to N$ is a diffeomorphism class of 
    $1$-dimensional oriented bordisms $X$ equipped with $\partial X \cong M^- \amalg N$
    and a relative integral cohomology class $\xi\in H^1(X, M\amalg N)$. 
    The composition of two morphisms $(X,\xi):M \to N$ and $(Y,\zeta):N \to L$ 
    is the morphism $(X\amalg_N Y,\chi)$ where $\chi$ is defined as the image of $(\xi,\zeta)$ under
    \[
        H^1(X,M\amalg N) \oplus H^1(Y,N\amalg L) \cong H^1(X \amalg_N Y, M \amalg N \amalg L) \to H^1(X \amalg_N Y, M \amalg L).
    \]
    The symmetric monoidal structure is defined by taking disjoint unions, the unit is the empty set.
\end{defn}

\begin{nrem}\label{nrem:ptga}
    Since the bordisms $X:M\to N$ are oriented $1$-manifolds we can canonically identify $H^1(X,M\amalg N)$ with $\IZ^{\pi_0 X}$.
    Therefore choosing $\xi$ is equivalent to labelling every connected component of $X$ by an integer.
    The composition adds integers of connected components that are joint in the process of glueing bordisms.
    Graphically this can be described as:
    \[
        \begin{tikzpicture}[yscale=0.75]
            \def\diam{0.05}
            \node[below] at (0,-.0) {$+$};
            \draw[fill] (0,0) circle (\diam);
            \draw [->] (0.1,0) to [out=0, in=180] node[above]{$2$} (1.9,1) ;
            \draw [<-] (1.9,0) to [out=180, in=180] node[left]{$-1$} (1.9,-1);
            \draw[fill] (2,1) circle (\diam);
            \draw[fill] (2,0) circle (\diam);
            \draw[fill] (2,-1) circle (\diam);
            \node[below] at (2,1) {$+$};
            \node[below] at (2,0) {$+$};
            \node[below] at (2,-1) {$-$};
            \node at (3,0) {$\circ$};
            \draw[fill] (4,1) circle (\diam);
            \draw[fill] (4,0) circle (\diam);
            \draw[fill] (4,-1) circle (\diam);
            \node[below] at (4,1-.0) {$+$};
            \node[below] at (4,-.0) {$+$};
            \node[below] at (4,-1.0) {$-$};
            \draw [->] (4.1,0) to [out=0, in=0] node[right]{$0$} (4.1,-1);
            \draw[fill] (6,1) circle (\diam);
            \draw[fill] (6,0) circle (\diam);
            \draw[fill] (6,-1) circle (\diam);
            \node[below] at (6,1-.0) {$+$};
            \node[below] at (6,-.0) {$+$};
            \node[below] at (6,-1.0) {$-$};
            \draw [->] (4.1,1) to [out=0, in=180] node[below]{$1$} (5.9,1);
            \draw [<-] (5.9,0) to [out=180, in=180] node[left]{$1$}  (5.9,-1);
            \node at (7,0) {$=$};
            \draw[fill] (8,0) circle (\diam);
            \node[below] at (8,-.0) {$+$};
            \draw [->] (8.1,0) to [out=0, in=180] node[above]{$3$} (9.9,1);
            \draw [<-] (9.9,0) to [out=180, in=180] node[left]{$1$} (9.9,-1);
            \draw [<-] (8.699,-1) to node[above]{$-1$} (8.7,-1);
            \draw (8.7,-0.6) circle (0.4);
            \draw[fill] (10,1) circle (\diam);
            \draw[fill] (10,0) circle (\diam);
            \draw[fill] (10,-1) circle (\diam);
            \node[below] at (10,1-.0) {$+$};
            \node[below] at (10,-.0) {$+$};
            \node[below] at (10,-1.0) {$-$};
        \end{tikzpicture}
    \]
    We write $*_+$ for the object defined by one positively oriented point and $\ga:*_+ \to *_+$ for its automorphism defined by the trivial bordism labelled by the integer $1$.
\end{nrem}

\begin{defn}
    Let $\CoN \subset \CoZ$ denote the subcategory
    that contains all objects, but only those morphisms that are 
    labelled by non-negative integers under the identification 
    in \ref{nrem:ptga}.
\end{defn}

\begin{nrem}\label{nrem:choosesection}
    We will now show that $\CoZ$ is equivalent to the homotopy category of $\Bor(S^1)$
    and that $\CoN$ is equivalent to the homotopy category of $\Borm$.
    Recall that in the process of defining $h_1 X$ for a complete Segal space $X:\gD^{op} \to \Spc$ we had to choose a section $o:\pi_0 X_0 \to X_0$.
    For $\Bor(S^1)$ we may choose $o$ to take values $(t,(M,\gp))$ such that 
    $\gp:M \to S^1$ only hits the base-point of $S^1$.
\end{nrem}

\begin{lem}\label{lem:h1BfrS1}
    There is a commutative diagram of symmetric monoidal functors
    \[
    \begin{tikzcd}
        h_1\Borm \ar[d, "G'"] \ar[r, "h_1L"] & h_1\Bor(S^1) \ar[d, "G"] \ar[r] 
        & h_1\Bor \ar[d, "F"]\\
        \CoN \ar[r] & \CoZ \ar[r] & \mi{Cob}_1
    \end{tikzcd}
    \]
    where the vertical functors are equivalences.
\end{lem}
\begin{proof}
    The rightmost vertical functor is the symmetric monoidal equivalence 
    from \ref{lem:hBord=Cob}. To prove the the lemma we have to provide
    compatible lifts $G$ and $G'$.
    
    For $G$ this means that we have to give, for each 
    $[X,\gp]:(M, \gp_{|M}) \to (N, \gp_{|N})$ a class $\ga \in H^1(X, M\amalg N)$.
    Because of the choice we made in \ref{nrem:choosesection} we may assume 
    that both $\gp_{|M}$ and $\gp_{|N}$ are the constant maps 
    to the basepoint $* \in S^1$. Therefore, the pullback of the canonical generator
    $[S^1] \in H^1(S^1,*)$ gives a well-defined class 
    $\ga := \gp^*(S^1) \in H^1(X, M \amalg N)$.
    This construction is compatible with gluing and disjoint union and 
    therefore yields a symmetric monoidal functor $G$.
    Moreover, the assignment $[\gp] \mapsto \gp^*[S^1]$ defines a bijection
    between $\pi_0\Map((X, M\amalg N), (S^1,*))$ and $H^1(X, M \amalg N)$ 
    for any bordism $X$ and hence $G$ is an equivalence of categories.
    
    To obtain the functor $G'$, first consider the composite $G \circ h_1L$.
    Concretely, let $(W, A): M \to N$ be some marked bordism. 
    Then $L(W,A) = (W, \gp)$ is equipped with a labelling
    such that $\gp:W \to S^1$ winds around the circle once for each marking.
    This is done compatibly with the orientation and hence every connected 
    component $W_0 \subset W$ will be labeled in $(G\circ h_1L)(W,A)$
    by the number of elements of $W_0 \cap A$.
    Since this is non-negative it lies in the subcategory $\CoN$
    and the functor $G \circ h_1L$ can be factored through a unique 
    symmetric monoidal functor $G': h_1 \Borm \to \CoN$.
    This functor is an equivalence of categories because up to diffeomorphism
    the marking of a bordism is uniquely determined by the number 
    of marked points in each connected components.
\end{proof}

Note that warning \ref{warning} applies to the first line of the following proposition
as we use the conjectural cobordism hypothesis with singularities.

\begin{proposition}\label{prop:1class}
    There are compatible isomorphisms of commutative monoids
    \[
    \begin{tikzcd}
        \eTTo \ar[r, "\cong"] \ar[d]
        & \Sc(\mi{Cob}_1^\IN) \ar[r, "\cong"] \ar[d]
        & \Mfd_1^\IN/\Diff^+ \ar[r, "\cong"] \ar[d]
        & \IN[x] \ar[d] \\
        \TTo \ar[r, "\cong"] 
        & \Sc(\mi{Cob}_1^\IZ) \ar[r, "\cong"] 
        & \Mfd_1^\IZ/\Diff^+ \ar[r, "\cong"] 
        & \IN[x^{\pm 1}] 
    \end{tikzcd}
    \]
    under which the tracelike transformation $\gT^{k_1,\dots,k_n}$ is sent to 
    $\sum_{i=1}^n x^{k_i}$.
\end{proposition}
\begin{proof}
    We begin by constructing the isomorphisms in the second line,
    the first line is then obtained similarly.
    As a consequence of lemma \ref{lem:h1BfrS1}, the $h_1\dashv N$
    adjunction \ref{cor:h1N}, the cobordism hypothesis \ref{cor:cobhyp-A}
    and lemma \ref{lem:recov} we have isomorphisms natural
    in $\mcC \in \Cat_1^\ot$:
    \begin{align*}
        \pi_0 \mi{Fun}_1^\ot(\CoZ, \mcC) 
        &\stackrel{\text{\ref{lem:h1BfrS1}}}{\cong} \pi_0\mi{Fun}_1^\ot(h_1\Bor(S^1), \mcC) \stackrel{\text{\ref{cor:h1N}}}{\cong} \pi_0\Funo{\Bor(S^1), N\mcC} \\
        &\stackrel{\text{\ref{cor:cobhyp-A}}}{\cong}  \pi_0 \mcA^\fd(N\mcC) 
        \stackrel{\text{\ref{lem:recov}}}{\cong} A^\fd(\mcC).
    \end{align*}
    The object $*_+$ and automorphism $\ga$ from \ref{nrem:ptga} define an element of $A^\fd(\CoZ)$; this is precisely the element that corresponds to $\id_{\CoZ}$ when we apply the above bijection in the case $\mcC=\CoZ$.

    We now apply the $1$-categorical coYoneda lemma 
    for the $1$-category $h_1\Cat_1^\ot$. 
    (See theorem \ref{thm:coYoneda} where we recall the $\infty$-categorical version.)
    Recall that in this category objects are symmetric monodial
    $1$-categories and morphisms are \emph{isomorphism classes} 
    of symmetric monoidal functors. 
    So the above $A^\fd(\mcC)$ can be described as 
    $\Hom_{h_1\Cat_1^\ot}(\CoZ, \mcC)$ in this category.
    Therefore the coYoneda lemma implies that sending $T$ to 
    $T_{\CoZ}(*_+,\ga)$ defines a bijection:
    \[
        \TTo = \Hom_{\mi{Fun}(h_1\Cat_1^\ot,\Sets)}(A^\fd, \Sc)
            \cong \Hom_{\mi{Fun}(h_1\Cat_1^\ot,\Sets)}(\Hom_{\Cat_1^\ot}(\mi{Cob}_1^\IZ,\blank), \Sc) 
            \cong \Sc(\mi{Cob}_1^\IZ).
    \]
    This is an isomorphism of monoids since 
    the monoid structure on $\TTo$ is defined by pointwise multiplication.

    The scalars of $\CoZ$ are diffeomorphism classes of closed oriented $1$-manifolds labelled by integers, this is precisely how the set $\Mfd_1^\IZ/\Diff^+$ was defined in the introduction.
    Recall that $\IN[x^{\pm 1}]$ is the free commutative monoid on the set 
    $\{\dots, x^{-1}, x^0, x^1, x^2, \dots\} \cong \IZ$.
    We define a homomorphism $\IN[x^{\pm 1}] \to \Mfd_1^\IZ/\Diff^+$ by sending 
    $x^k$ to the circle labelled by $k$.
    This is a bijection since all closed $1$-manifolds 
    can be written as the disjoint union of circles.
    
    This establishes the isomorphisms in the second line of the proposition.
    The first line is obtained similarly, except that we use the cobordism hypothesis
    with singularities \ref{conj:singcobhyp} instead of \ref{cor:cobhyp-A}.

    To complete the proof we need to compute the value of the tracelike transformation $\gT^{k_1,\dots,k_n}$ on $(*_+, \ga)$.
    In other words, we have to compute $\Tr(\ga^{k_1}) \circ \dots \circ \Tr(\ga^{k_n})$ using the classical definition of the trace.
    Using the evaluation and coevaluation provided in \ref{lem:BordHasDuals}
    one sees that the trace of a bordism
    $(X,\partial X \cong M \amalg M^-)$ is given by gluing
    the two boundaries of $X$.
    See also \cite{ST12} for a different perspective on this.
    In the case at hand we compute:
    \[
        \begin{tikzpicture}[scale = 0.9]
            \def\s{0}
            \def\d{0.5} 
            \def\e{7} 
            \def\diam{0.05}
            \foreach \x in {0,\d,\d+2,2*\d+2,\e,\d+\e,\d+\e+2,2*\d+\e+2}
            {
                \node[above] at (\x,1) {$+$};
                \node[above] at (\x,0) {$-$};
                \draw[fill] (\x,1) circle (\diam);
                \draw[fill] (\x,0) circle (\diam);
            }
            \foreach \x in {2*\d+3,3*\d+3,\e+2*\d+3,\e+3*\d+3}
            {
                \node[above] at (\x,1) {$-$};
                \node[above] at (\x,0) {$+$};
                \draw[fill] (\x,1) circle (\diam);
                \draw[fill] (\x,0) circle (\diam);
            }

            \draw [->] (0-\diam,0) to [out=180, in=180] node[left]{$0$} (0-\diam,1);
            \draw [->] (\d,1) to node[above]{$k_1$} (\d+2-\diam,1);
            \draw [<-] (\d+\diam,0) to [out=0, in=180] node[above]{$0$} (\d+2,0);
            \draw [<-] (2*\d+2,0) to [out=0, in=180] (2*\d+3-\diam,1);
            \draw [->] (2*\d+2,1) to [out=0, in=180] (2*\d+3-\diam,0);
            \draw [->] (3*\d+3+\diam,0) to [out=0, in=0] node[right]{$0$} (3*\d+3+\diam,1);
            
            \node at (\e/2+2.25,0.5) {$\dots$};
            
            \draw [->] (\e+0-\diam,0) to [out=180, in=180] node[left]{$0$} (\e+0-\diam,1);
            \draw [->] (\e+\d,1) to node[above]{$k_n$} (\e+\d+2-\diam,1);
            \draw [<-] (\e+\d+\diam,0) to [out=0, in=180] node[above]{$0$} (\e+\d+2,0);
            \draw [<-] (\e+2*\d+2,0) to [out=0, in=180] (\e+2*\d+3-\diam,1);
            \draw [->] (\e+2*\d+2,1) to [out=0, in=180] (\e+2*\d+3-\diam,0);
            \draw [->] (\e+3*\d+3+\diam,0) to [out=0, in=0] node[right]{$0$} (\e+3*\d+3+\diam,1);

            \def\k{14} 
            \node at (\k-1,0.5) {$=$};

            \draw [->] (\k-0.01,1) to node[above]{$k_1$} (\k,1);
            \draw (\k,0.5) circle (0.5);
            \node at (\k+1,0.5) {$\dots$};
            \draw [->] (\k+2-0.01,1) to node[above]{$k_n$} (\k+2,1);
            \draw (\k+2,0.5) circle (0.5);

            \end{tikzpicture}
    \]
    This verifies that $\gT^{k_1,\dots,k_n}(*_+,\ga)$ is a disjoint union of $n$ circles labelled by the integers $k_1, \dots, k_n$.
    Hence $\gT^{k_1,\dots,k_n}$ is sent to the polynomial 
    $x^{k_1} + \dots + x^{k_n}$ as claimed.
\end{proof}

\section[Infinity-categorical classification]{$\infty$-categorical classification}
In this section we combine the cobordism hypothesis with the Yoneda lemma 
and make the necessary computations to prove the main theorems.

\subsection{Applying the Yoneda lemma and the cobordism hypothesis}
\begin{nrem}
    We begin by recalling the `coYoneda lemma':
    using the Yoneda embedding $Y_\cC: \cC \to \sP(\cC)$ we can think of $Y_\cC(x) = \hom_\cC(\blank,x)$ as a functor from $\cC^{op}$ to $\Spc$.
    For $\cD:=\cC^{op}$  we then have $\hom_\cD(x, \blank): \cD \to \Spc$.
    The coYoneda lemma tells us how to compute natural transformations out of this functor.
    Recall that for two \icats\ $\cC$ and $\cD$ and functors $F,G:\cC\to \cD$ the space of natural transformations $F\Rightarrow G$ is
    \[
        \hom_{\Fun(\cC,\cD)}(F,G).
    \]
    Evaluating at a certain object $x\in \cC$ gives a functor $ev_x: \Fun(\cC,\cD) \to \cD$ and consequently a map $\hom_{\Fun(\cC,\cD)}(F,G) \to \hom_\cD(F(x), G(x))$.
\end{nrem}

\begin{thm}[coYoneda lemma, {\cite[Lemma 5.5.2.1]{LurHTT}}]\label{thm:coYoneda}
    For any functor $F:\cD \to \Spc$ and object $x\in \cC$ the following composition of evaluations is a natural equivalence
    \[
        \hom_{\Fun(\cD,\Spc)}( \hom_\cD(x, \blank), F) \xrightarrow{ev_x} \hom_\Spc( \hom_\cD(x,x), F(x)) \xrightarrow{ev_{\id_x}} F(x).
    \]
\end{thm}
We can now use the coYoneda lemma in conjunction with the one-dimensional
cobordism hypothesis to express the moduli spaces $\eTTi$ and $\TTi$
as the space of scalars of $\Bor(S^1)$ and $\Borm$, respectively.
Note, however, that to compute $\eTTi$ we use the
conjectural cobordism hypothesis with singularities,
so warning \ref{warning} applies.
\begin{cor}\label{cor:applyYoneda}
    There are equivalences of spaces 
    \[
        \TTi \simeq \Sc\left(\Bor(S^1)\right) 
        \qand
        \eTTi \simeq \Sc\left(\Borm\right).
    \]
\end{cor}
\begin{proof}
    By corollary \ref{cor:cobhyp-A} of the cobordism hypothesis \ref{thm:cobhyp}, 
    the functor $\mcA^\fd$ is naturally equivalent to 
    the corepresented functor $\hom_{\Catit}(\Bor(S^1), \blank)$.
    Then first equivalence now follows from the coYoneda lemma 
    \ref{thm:coYoneda} for $\cD=\Catit$, $x=\Bor(S^1)$, and $F=\Sc$.
    
    Similarly, the cobordism hypothesis with singularities \ref{conj:singcobhyp} 
    in dimension $1$ says that $\mcE^\fd$  is naturally equivalent to 
    the corepresented functor $\hom_{\Catit}(\Borm, \blank)$.
    Hence the second equivalence also follows from the coYoneda lemma.
\end{proof}

\subsection{Identification of the homotopy-type of the scalars of the bordism category}
We give a more geometric description of $\Sc(\Bor(S^1))$:
\begin{defn}\label{def:Mfd}
    For a topological space $X$, let $\Mfd_d^{X} \subset \Psi_d(X)$ 
    be the space of of closed $d$-manifolds submanifolds $M\subset \cube$
    equipped with an $X$-structure $\gp:M \to X$. 
\end{defn}

\begin{lem}\label{lem:ScMfd}
	There is an equivalence $\Sc(\Bor(X)) \simeq \Mfd_1^{X} $.
\end{lem}
\begin{proof}
    The monodial unit of $\Bor(X)$ 
    is given by the empty manifold, so by lemma \ref{lem:functors} 
    the space of scalars is equivalent to $\hom_{\Bor(X)}(\emptyset, \emptyset)$,
    which is indeed equivalent to $\Mfd_d^X$.
\end{proof}

\begin{defn}
    We write $\IT$ for $S^1$ with the usual group structure, 
    i.e.\ $\IT\cong U_1 \cong SO_2$.
    For a space $X\in \Spc$ we denote its \emph{free loop} space by
    \[
        \gL X := \Map(\IT,X).
    \]
    The group $\IT$ acts on this by precomposition and we denote the homotopy-orbits by
    \[
        (\gL X)_{h\IT} := (E\IT \times \gL X)/\IT.
    \]
\end{defn}

The space of closed manifolds is best understood as a free $E_\infty$-algebra
on the space of connected closed manifolds. Recall the following:
\begin{thm}\label{thm:BPQ}
    For a topological space $X$ the 
    underlying space of the free $E_\infty$-algebra on $X$ is 
    given by:
    \[
        \EPi(X) \simeq \coprod_{n\ge 0} X^n_{h\gS_n}.
    \]
\end{thm}
\begin{proof}
    This is a classical theorem by Segal: 
    his proof of the Barratt-Priddy-Quillen theorem in \cite[Proposition 3.5 and 3.6]{Seg74} 
    in fact shows that the unordered configuration space $\Conf_*(\IR^\infty;X)$ is a model for the free $\gC$-space on $X$.
    A modern reference is \cite[Proposition 3.1.3.13]{LurHA}.
\end{proof}

The following is fairly standard, see for instance \cite[Lemma 7.1]{Gia17},
but we give a proof for completeness sake.
\begin{lem}\label{lem:Mfd1X}
	For any topological space $X$ the space of closed $1$-manifolds with map to $X$ is:
	\begin{align*}
        \Mfd_1^{X} \simeq \EPi\left( (\gL X)_{h\IT} \right)
	\end{align*}
	Here, as in theorem \ref{thm:BPQ}, 
	$\EPi(Y)$ denotes the free $E_\infty$-algebra on a space $Y$.
\end{lem}
\begin{proof}
	The space $\Mfd_1$ has a natural $\IN$-grading induced by the number of connected components.
	All closed $1$-manifolds are disjoint unions of circles, 
	so the $n$th level of the grading is space of submanifolds of $\cube$
	that are abstractly diffeomorphic to $(S^1)^{\amalg n}$.
	In the case $X=*$ this space is equivalent to the classifying space 
	$B\Diff^+((S^1)^{\amalg n})$ by \cite[Corollary B.5]{SP17}.
	A similar argument shows that for general $X$ it is 
	the homotopy-orbits of $\Map((S^1)^{\amalg n}, X)$
	with respect to the action of $\Diff^+((S^1)^{\amalg n})$:
	\[
        \Map\left((S^1)^{\amalg n}, X\right)_{h\Diff^+((S^1)^{\amalg n})}.
	\]
	The group $\Diff^+((S^1)^{\amalg n})$ can be decomposed as 
	a wreath product $\gS_n \wr \Diff^+(S^1)$
	acting component-wise $ \Map(S^1,X)^n$.
	Since we are working with homotopy actions, we may replace 
    the group $\Diff^+(S^1)$ by the equivalent group $\IT$.
	Homotopy-orbits with respect to the action of a wreath product can be computed as
	\[
        \left(\Map(S^1,X)^n \right)_{h(\gS_n \wr \IT)} \simeq \left(\left(\Map(S^1,X)_{h\IT}\right)^n\right)_{ h\gS_n}.
	\]
	Putting all the parts of the $\IN$-grading back together we get
	\[
		\Mfd_1^{X} \simeq \coprod_{n \ge 0}  \left(\left(\Map(S^1,X)_{h\IT}\right)^n\right)_{ h\gS_n} \simeq \EPi\left( \Map(S^1,X)_{h\IT} \right)
	\]
    as claimed.
    Here the last equivalence is as in theorem \ref{thm:BPQ}.
\end{proof}

\begin{lem}\label{lem:gL(S^1)}
    In the case $X=S^1$ that is most relevant to us, we compute:
	\[
        \gL(S^1)_{h \IT} \simeq (S^1 \times BS^1) \amalg \coprod_{k \in \IZ\setminus\{0\}} B(\IZ/k\IZ).
	\]
\end{lem}
\begin{proof}
    The space $\Map(S^1,S^1)$ has as connected components the spaces $\Map^k(S^1,S^1)$ of maps with winding number $k$ for $k\in \IZ$.
    We need to compute $\Map^k(S^1,S^1)_{h \IT}$ for all $k \in \IZ$.
    Let $X_k$ be the space $S^1 \subset \IR^2$ with the action of $\IT$ defined by $\gl.\zeta := \gl^k \cdot \zeta$.
    There is an $\IT$-equivariant embedding $\gi:X_k \to \Map^k(S^1,S^1)$ that identifies $X_k$ with the space of degree $k$ maps $S^1 \to S^1$ of constant speed.
    Non-equivariantly $\gi$ is a homotopy equivalence with the inverse $\Map^k(S^1,S^1) \to S^1$ given by evaluation on the base-point.
    Therefore $\gi$ is a Borel weak equivalence and induces an equivalence on the homotopy-orbits:
    \[
        \Map(S^1,S^1)_{h \IT} \simeq \coprod_{k\in \IZ} (X_k)_{h \IT}.
    \]
    To compute $(X_k)_{h \IT}$ for $k \neq 0$ observe that $X_k$ can be thought of as the quotient $S^1/(\IZ/k\IZ)$ with $\IT$-action induced from the standard action of $\IT$ on $S^1$.
    Therefore 
    \begin{align*}
        (X_k)_{h \IT} &= (X_k \times E\IT)/\IT \cong ((S^1/(\IZ/k\IZ)) \times E\IT)/\IT \\
                       &\cong \left((S^1 \times E\IT)/\IT \right)/(\IZ/k\IZ)
    \end{align*}
    The space $(S^1 \times E\IT)/\IT$ is homeomorphic to $E\IT$ and therefore contractible.
    The action of $\IZ/k\IZ$ is free and hence $(X_k)_{h \IT}$ is a model for $B(\IZ/k\IZ)$.

    The remaining case $k=0$ is easy: the space $X_0$ is $S^1$ with the trivial $\IT$-action, therefore the homotopy fixed-points decompose as $(X_0)_{h \IT} \simeq S^1 \times B\IT$.
\end{proof}

\subsection{The space of marked $1$-manifolds}

\begin{lem}\label{lem:ScBfrm}
    There is an equivalence:
    \[
        \Sc(\Borm) \simeq \EPi \left( BS^1 \amalg \coprod_{k = 1}^\infty B(\IZ/k\IZ) \right) .
    \]
\end{lem}
\begin{proof}
    Just like in lemma \ref{lem:ScMfd} we see that the space of scalars $\Sc(\Borm)$
    is equivalent to the space of closed marked oriented $1$-dimensional submanifolds
    of $\cube$.
    By an argument as in proposition \ref{prop:Bordm-CSS} this space is homeomorphic 
    to a Borel construction:
    \[
        \Mfd_1^{\mi{or},m} \cong \coprod_{n=0}^\infty 
        \left(\Emb((S^1)^{\amalg n}, \cube) \times \Conf_*((S^1)^{\amalg n})\right)
        /\Diff^+((S^1)^{\amalg n})
    \]
    Here the diffeomorphism group $\Diff^+(M)$ acts diagonally on 
    $\Conf_*(M) \subset \coprod_{k \ge 0} M^k/\gS_k$.
    The configuration space has the property that there is a canonical homeomorphism
    $\Conf_*(M \amalg M') \cong \Conf_*(M) \times \Conf_*(M')$
    for any two manifolds $M$ and $M'$. In particular we have 
    $\Conf_*((S^1)^{\amalg n}) \cong \prod_{i=1}^n \Conf_*(S^1)$.
    We can therefore argue just like in lemma \ref{lem:Mfd1X} to see that $\Mfd_1^{\mi{or},m}$
    is equivalent to the underlying space of a free $E_\infty$-algebra:
    \[
        \Mfd_1^{\mi{or},m} \simeq \EPi\left(
        \Conf_*(S^1)_{h\IT}
        \right).
    \]
    We are now left to compute the homotopy $\IT$-orbits of $\Conf_*(S^1)$.
    To do so, we can decompose the configuration space by the cardinality
    of the configuration. In the $k=0$ case we have only the empty configuration,
    so $\Conf_0(S^1) = \{\emptyset\}$ is a point and the homotopy orbits are $B\IT$.
    For $k>0$ let, as before, $X_k$ be the space $S^1 \subset \IC$ with the action of $\IT$
    defined by $\gl.\zeta := \gl^k \cdot \zeta$.
    There is an equivariant map $i: X_k \to \Conf_0(S^1)$ which sends 
    $\zeta$ to the subset $\{\xi\;|\; \xi^k = \zeta\} \subset S^1$.
    This map identifies $X_k$ with the subspace of `equally spaced' configurations
    of $k$ points on the circle and therefore is an equivalence. 
    As in lemma \ref{lem:gL(S^1)} we have that $(X_k)_{h\IT} \simeq B\IZ/k$
    and hence the result follows.
\end{proof}

\subsection{Proof of theorem \ref{thm:As}}

\begin{thm}[Theorem A]\label{thm:Ap}
    There is a commutative diagram of spaces:
    \[
        \begin{tikzcd}[row sep=tiny, column sep=small]
            \eTTi \ar[dd] \ar[rr, "\simeq"] \ar[rd] & &
            \EPi(BS^1 \amalg \coprod_{k\ge 1}B\IZ/k\IZ)\ar[dd] \ar[rd] & \\
            & \TTi \ar[dd] \ar[rr, "\simeq", near start] & &
            \EPi(S^1 \times BS^1 \amalg 
                    \coprod_{k\in \IZ\setminus\{0\}}B\IZ/k\IZ)\ar[dd] \\
            \eTTo \ar[rr, "\cong"] \ar[rd] & & \IN[x] \ar[rd] & \\
            & \TTo \ar[rr, "\cong"] & & \IN[x^{\pm 1}]
        \end{tikzcd}
    \]
    where the horizontal maps are equivalences.
    Moreover, the vertical maps identify the sets in the bottom layer
    as the set of connected components of the top layer: 
    $\eTTo \cong \pi_0 \eTTi$ and $\TTo \cong \pi_0 \TTi$.
\end{thm}
\begin{proof}
    The bijections in the bottom row are discussed in proposition \ref{prop:1class} about the classification of the $1$-categorical tracelike transformations.
    For the top row corollary \ref{cor:applyYoneda} identifies $\TTi$ and $\eTTi$
    with $\Sc(\Bor(S^1))$ and $\Sc(\Borm)$, respectively.
    These spaces of scalars are then computed in lemmas 
    \ref{lem:ScMfd}, \ref{lem:Mfd1X}, \ref{lem:gL(S^1)}, and \ref{lem:ScBfrm}.

    Since the $\infty$-categorical classification followed the same steps as the $1$-categorical one, the left-hand square commutes.
    By \ref{nrem:mcL-commutes} the top-right map can be understood
    as $\Sc(L):\Sc(\Bor(S^1)) \to \Sc(\Borm)$.
    Combining the definition of $L$ with the equivalences in lemmas
    \ref{lem:Mfd1X}, \ref{lem:gL(S^1)}, and \ref{lem:ScBfrm} 
    it follows that it is the inclusion
    $BS^1 \cong \{0\} \times BS^1 \inj BS^1$ on the first component
    and the identity on $B\IZ/k \to B\IZ/k$ for $k>0$.
    The right-most vertical map is $\Sc(\Bor(S^1)) \to \Sc(\CoZ)$, 
    which is a bijection on connected components because
    $\CoZ \simeq h_1 \Bor(S^1)$.
    This implies that $\pi_0\eTTi \cong \eTTo$ and by the same argument
    $\pi_0\TTi \cong \TTo$.
\end{proof}
Note that the vertical map sends $S^1\times BS^1$ to $\{x^0\}$ and $B\IZ/k\IZ$ to $\{x^k\}$.

\begin{cor}[{Corollary \ref{cor:Bs}, generalising \cite[Théorème 3.18]{TV15}}] \label{cor:Bp}
    The space of (restricted) $\infty$-cate\-gorical tracelike transformations 
    that act as the $1$-categorical trace $\Tr$ on homotopy categories is contractible.
\end{cor}
\begin{proof}
    Under the equivalences of theorem A, the preimage of $\Tr$ under $\eTTi \to \eTTo$
    is the connected component 
    \[
        B(\IZ/1\IZ) \ \subset \ B\IT \amalg \coprod_{k=1}^\infty B\IZ/k\IZ
        \ \subset \ \EPi\left(B\IT \amalg \coprod_{k=1}^\infty B\IZ/k\IZ\right).
    \]
    But $\IZ/1\IZ$ is the trivial group so this connected component is contractible.
    The same argument applies to the restricted case $\TTi \to \TTo$,
    recovering To\"en an Vezzosi's \cite[Théorème 3.18]{TV15}.
    (Note that by lemma \ref{lem:functors} the definition of $\mcA$ 
    really recovers that of To\"en and Vezzosi.)
\end{proof}

\begin{cor}[Corollary \ref{cor:Cs}]\label{cor:Cp}
    Any $\infty$-categorical tracelike transformation
    whose value on the category of complex vector spaces agrees 
    with the trace from linear algebra is canonically equivalent
    to the $\infty$-categorical trace. 
\end{cor}
\begin{proof}
    By corollary \ref{cor:Bp} the connected component of $\Tr$ 
    in $\eTTi$ is contractible.
    Since we have a full classification of tracelike transformations 
    it is enough to show that for any $\IN$-polynomial 
    $\sum_{i=1}^n x^{k_i}$ the map
    \[
        \gT_{\Vect_k}^{k_1,\dots,k_n}: E^\fd(\Vect_\IC) \to \Sc(\Vect_\IC) \cong \IC, \quad (V,f:V \to V) \mapsto \Tr(f^{k_1}) \cdots \Tr(f^{k_n})
    \]
    is the trace from linear algebra only if $n=1$ and $k_1 = 1$.

    For this purpose consider the $2\times 2$ diagonal matrix $A_\gl$
    with entries $1$ and $\gl\in \IC$:
    \[
        \gT_{\Vect_\IC}^{k_1,\dots,k_n}(\IC^2,A_\gl) = (1 + \gl^{k_1}) \cdots (1 + \gl^{k_n}).
    \]
    Comparing coefficients we see that this can only be equal to $\Tr(A_\gl)=1+\gl$ for all $\gl \in \IC$ if $n=1$ and $k_1=1$.
\end{proof}

    Note that by the same argument corollary \ref{cor:Cs} 
    also holds for restricted tracelike transformations.

\begin{remark}
    The statement of corollary \ref{cor:Cs} remains true if we replace
    $\IC$ by any other infinite field.
    It then however is not sufficient to look only at the matrices $A_\gl$.
    If, for instance, $k$ is of characteristic $3$, then $\gT_{\Vect_k}^{1,0,0}(k^2, A_\gl) = (1 + \gl)\cdot 2 \cdot 2 = 1+\gl$.
    This problem can be solved by considering larger matrices.
    Note that corollary \ref{cor:Cs} is not true for finite fields: 
    over $\mathbb{F}_p$ the tracelike transformations $\gT^{p}$ 
    and $\gT^1=\Tr$ agree. 
\end{remark}

\subsection{The action of $\IN$ and the proof of corollary \ref{cor:Ds}}
We would like to give a purely categorical characterisation of the trace 
among the other tracelike transformations that does not rely on a pre-defined 
notion of trace.
To define the notion of \emph{generating tracelike transformation}
we first need to specify the $\IN$-action outlined in the introduction.

\begin{defn}\label{defn:h-mcP}
    The multiplicative monoid $(\IN,\cdot)$ acts on the sets $\eTTo$ and $\TTo$
    by $(P^n T)(e) := T(e^n)$.
\end{defn}

It is clear that $P^n \circ P^m = P^{nm}$, but we need to briefly check
that $P^n T$ is indeed still conjugation invariant:
\[
    (P^n T)(g^{-1} \circ f \circ g) = T((g^{-1} \circ f \circ g)^n) 
    = T(g^{-1} \circ f^n \circ g) = T(f^n) = (P^n T)(f).
\]
We will now lift the maps $P^n$ to maps $\mcP^n:\pTTi \to \pTTi$ 
of the moduli space of $\infty$-categorical tracelike transformations.
Once could assemble these maps into a coherent action of $(\IN, \cdot)$,
but that is not necessary for our purposes.
Instead we will just check that the $\mcP^n$ induce $P^n$ on
$\pi_0 \pTTi \cong \pTTo$.

\begin{construction}
    Fix $n \in \IN$ and some complete Segal space $X_\cd$.
    There is a diagonal map 
    \[
        \gD_{X_1}^n: X_0 \times_{(X_0)^2}X_1 \to 
        X_0 \times_{(X_0)^2} (X_1 \times_{X_0} \dots \times_{X_0} X_1),
    \]
    where as usual all pullbacks are computed in the infinity category of spaces.
    We can pick a homotopy inverse to the Segal map
    $X_n \to X_1 \times_{X_0} \dots \times_{X_0} X_1$
    and construct the composite
    \[
        p^n_X: \mcE(X) = X_0 \times_{(X_0)^2}X_1 \xrightarrow{\gD_{X_1}^n}
        X_0 \times_{(X_0)^2} (X_1 \times_{X_0} \dots \times_{X_0} X_1)
        \xleftarrow{\simeq}  X_0 \times_{(X_0)^2} X_n
        \xrightarrow{d} X_0 \times_{(X_0)^2} X_1. 
    \]
    Here $d:X_n \to X_1$ is the face operator coming from the long edge $\{0,n\} \subset [n]$.
    This defines a natural transformation $p^n:\mcE \Rightarrow \mcE$
    that commutes with the projection $\mcE(X) \to X_0$.
    Given a symmetric monoidal structure on $X$ the transformation $p^n$
    therefore preserves the subspace $\mcE^\fd(X) \subset \mcE(X)$.
    We can hence define a map
    \[
        \mcP^n:\eTTi \to \eTTi, \quad 
        (T:\mcE^\fd \Rightarrow \Sc) \mapsto (T\circ p^n:\mcE^\fd \Rightarrow \mcE^\fd \Rightarrow \Sc).
    \]
\end{construction}

\begin{lem}
    For $T \in \eTTi$ and $n \in \IN$ the tracelike transformation $\mcP^n T$
    acts as $(\mcP^n T)(e) = T(e^n)$ on the homotopy category.
\end{lem}
\begin{proof}
    The inverse of the Segal map composed with the face map:
    \[
        X_1 \times_{X_0} \dots \times_{X_0} X_1
        \xleftarrow{\simeq}  X_n
        \xrightarrow{d} X_1
    \]
    is given by $(f_1, \dots, f_n) \mapsto f_n \circ \dots \circ f_1$ on the homotopy category.
    Therefore $p_X^n:\mcE(X) \to \mcE(X)$ yields
    \[
        E(h_1X) \to E(h_1X), \quad (x,e) \mapsto (x,e^n).
    \]
    Precomposing with this induces $P^n:\eTTo \to \eTTo$ as described 
    in definition \ref{defn:h-mcP}.
\end{proof}

We can now make the definition from the introduction precise:
\begin{defn}
    A tracelike transformation $T \in \eTTi$ is 
    \emph{generating}
    if the monoid $\pi_0 \eTTi$ is generated by the set
    $\{[\mcP^n(T)]\;|\; n \in \IN\}$.
    Equivalently, $T$ is generating if for every other $S \in \eTTi$
    there are non-negative integers $k_1, \dots, k_n$ such that
    $S$ is equivalent to $\mcP^{k_1}(T) \cdots \mcP^{k_n}(T)$.
\end{defn}

\begin{cor}[Corollary \ref{cor:Ds}]\label{cor:Dp}
    The space of generating tracelike transformations is contractible
    and its image in $\eTTo$ is the classical trace $\Tr$.
\end{cor}
\begin{proof}
    By theorem \ref{thm:Ap} we know that $\pi_0 \eTTi \cong \eTTo$,
    and by corollary \ref{cor:Bp} the fibre of $\Tr$ under the projection
    $\eTTi \to \eTTo$ is contractible.
    It will therefore suffice to show that that $\Tr = \gT^{1}$ is the unique
    generating $1$-categorical tracelike transformation in $\eTTo$.
    
    The action of $P^m$ on $\gT^{k_1,\dots,k_n}$ is given by
    \[
        P^m(\gT^{k_1,\dots,k_n})(e) = \gT^{k_1,\dots,k_n}(e^m) 
        = \Tr((e^m)^{k_1}) \circ \dots \circ \Tr((e^m)^{k_n})
        = \gT^{mk_1, \dots, mk_n}(e).
    \]
    Therefore, under the isomorphism $\eTTo \cong \IN[x]$ of proposition \ref{prop:1class},
    $P^n$ acts on $\IN[x]$ as $x^m \mapsto x^{nm}$.
    It is clear that $x$ is the only element $y \in \IN[x]$ such that
    $\{P^n(y)\;|\;n \in \IN\}$ generates the monoid $\IN[x]$,
    and hence the same holds for $\Tr \in \eTTo$.
\end{proof}

\section[Application: the cyclic group action on Tr(fk) and THH]{Application: the cyclic group action on $\Tr(f^k)$ and THH}

In this section we study the other connected components of the moduli space 
$\eTTi$ of tracelike transformations, which we computed in Theorem \ref{thm:As}.
We will see that while $(e \mapsto \Tr(e))$
uniquely lifts to a $\infty$-categorical tracelike transformation, 
$(e \mapsto \Tr(e^k))$ has `$B(\IZ/k\IZ)$-many lifts'.
This induces a coherent $\IZ/k\IZ$-action on $\Tr(e^k)$.
As an example we will look at the derived Morita category
where the trace is given by topological Hochschild homology (THH).

\begin{construction}
    For every symmetric monodial \icat\ $\cC$ there is a canonical map
    $
        \eTTi \to \Map(\mcE^\fd(\cC),\Sc(\cC))
    $
    and by adjunction a map $\mcE^\fd(\cC) \to \Map(\eTTi, \Sc(\cC))$.
    Given a dualisable object $x \in \cC$ we can moreover compose this
    with the canonical map $\End_\cC(x) \to \mcE^\fd(\cC)$ to obtain:
    \[
        \theta: \End_\cC(x) \to \Sc(\cC)^{\eTTi} .
    \]
    Under the equivalence of theorem \ref{thm:As} we in particular have
    \[
        \theta_0: \End_\cC(x) \to \Sc(\cC)^{B\IT} 
        \qand
        \theta_k: \End_\cC(x) \to \Sc(\cC)^{B\IZ/k\IZ} 
        \text{ for } k >0 .
    \]
\end{construction}

\begin{nrem}
    As part of theorem \ref{thm:As} we know that the effect of 
    $\theta_k$ on the homotopy category is:
    \[
        \theta_k(x,e) = \Tr(e^k).
    \]
    Since giving a map $B\IZ/k\IZ \to \c{E}$ into an $\infty$-groupoid $\c{E}$
    is equivalent to giving an object of $\c{E}$ with a coherent 
    $\IZ/k\IZ$-action, the above construction shows that $\Tr(e^k)$ 
    comes with a natural action of $\IZ/k\IZ$ when thought of as an object in the
    $\infty$-groupoid $\Sc(\cC)$.
    For $k=0$ we have that $\Tr(e^0) = \Tr(\id_x)$ and the corresponding
    connected component of $\eTTi$ is $B\IT$, so $\Tr(\id_x)$ has 
    an action of the circle group $\IT$.
\end{nrem}
    
Note that all of these actions are trivial when $\cC$ is a $1$-category, 
since then $\Sc(\cC)$ is just a set and hence discrete as an $\infty$-groupoid.
We can therefore not observe the group actions in the category of vector spaces.
Instead we consider the following example: 
\begin{example}
    Consider the stable Morita category $\cc{Morita}$ as defined 
    in \cite[4.8.4.9]{LurHA} where objects are $E_1$-ring spectra $A,B,C, \dots$,
    morphisms $A \to B$ are $(A,B)$-bimodule spectra and composition
    of morphisms $M:A \to B$ and $N:B \to C$ is given by the relative
    tensor product $(M \ot_B N):A \to C$.
    This category is symmetric monoidal with respect to the 
    smash product $\ot = \wedge$ on spectra.
    
    In this category all objects $A$ are dualisable with dual $A^{op}$.
    Evaluation and coevaluation are given by $A$ thought of as a
    $(A^{op} \ot A, \IS)$ and $(\IS, A \ot A^{op})$-bimodule, respectively.
    The trace of an $(A,A)$-bimodule $M:A \to A$ is therefore given by 
    the relative tensor product:
    \[
        \Tr(M) \simeq A \ot_{A \ot A^{op}} M \ot_{A \ot A^{op}} A
        \simeq A \ot_{A \ot A^{op}} M.
    \]
    This, by definition, is the \emph{topological Hochschild homology}
    $\op{THH}(A; M)$ of $A$ with coefficients in $M$.
    
    The scalars of $\cc{Morita}$ form the $\infty$-groupoid of 
    $(\IS, \IS)$-bimodules, i.e.\ the maximal subgroupoid 
    of the category of spectra: $\Sc(\cc{Morita}) \simeq \Sp^\simeq$.
    Our construction of $\theta_k$ induces a natural
    $\IZ/k\IZ$-action on $\op{THH}(A; M \ot_A \dots \ot_A M)$
    and a natural $\IT$-action on $\op{THH}(A) = \op{THH}(A; A)$.
\end{example}

\begin{remark}
    There is a well-know $\IT$-action on $\op{THH}(A)$,
    which we expect to be equivalent to the above one. 
    Let us briefly sketch how one could go about proving this:
    
    Given a ring spectrum $A$, Scheimbauer in her thesis \cite{Sch14}
    uses factorization homology to construct a symmetric monoidal 
    functor $\mcZ: \Bor \to \c{Morita}$ that sends the positive point to $A$.
    We can compose this with the functor that forgets markings 
    to obtain $\mcZ': \Borm \to \Bor \to \c{Morita}$.
    By naturality, the $\IT$-action on $\gt_0(A)$ then has to be given by
    $\mcZ': B\IT \subset \Sc(\Borm) \to \Sc(\c{Morita})$.
    From Scheimbauer's construction we see that this is 
    the usual action of $\IT$ on the factorization homology $\int_{S^1} A$ 
    via $\IT \curveto S^1$.
    The claim therefore follows from the folklore theorem 
    that there is a $\IT$-equivariant equivalence 
    $\int_{S^1} A \simeq \op{THH}(A)$, see for instance \cite{AMGR17}.
\end{remark}

\begin{remark}
    Lindenstrauss and McCarthy construct a $\IZ/n\IZ$-equivariant 
    spectrum $U^n(F,P)$ for every ring spectrum $F$ and bimodule $P$, 
    see \cite[Definition 4.1]{LMc12}. By \cite[Lemma 2.6]{LMc12} 
    this spectrum is non-equivariantly equivalent to 
    $\op{THH}(F; P \ot_F \dots \ot_F P)$.
    We expect that the this induces the same $\IZ/n$-action 
    as one obtains form the tracelike transformation $e \mapsto \Tr(e^n)$
    as described above. 
    One could also attempt to prove this in a similar way 
    as indicated for $\gt_0$. However, one would now need
    a version of stratified factorization homology to construct
    the relevant functor $\Borm \to \c{Morita}$.
\end{remark}



\bibliography{literature}{}
\bibliographystyle{alpha}
 
\end{document}